\documentclass[12pt,reqno]{amsart}

\usepackage{graphicx,subfigure}
\usepackage{hyperref}
\hypersetup{colorlinks=true, citecolor=blue, linkcolor=red}

\usepackage{appendix}

\numberwithin{equation}{section} \numberwithin{figure}{section}
\numberwithin{table}{section} 
{ \theoremstyle{plain}
\newtheorem{theorem}{Theorem}[section]
\newtheorem{proposition}[theorem]{Proposition}
\newtheorem{lemma}[theorem]{Lemma}
}

{ \theoremstyle{definition}
\newtheorem{definition}[theorem]{Definition}

\newtheorem{remark}[theorem]{Remark}

}

\begin{document}

\title[Large solutions to the Gelfand-Liouville problem]{Matched asymptotics for large solutions to the Gelfand-Liouville problem in two-dimensional, doubly connected domains}
%{Heteroclinic orbits in slow-fast Hamiltonian systems with slow
%manifold bifurcations}

%\author[Schecter]{Stephen Schecter}
%\address{
%Department of Mathematics \\
%North Carolina State University \\
%Box 8205 \\
%Raleigh, NC 27695 USA \\
%919-515-6533 } \email{schecter@math.ncsu.edu}

\author[Sourdis]{Christos Sourdis}
\address{
Department of mathematics\\
University of Ioannina \\
Ioannina, Greece.} \email{csourdis@tem.uoc.gr}

%\date{September 2, 2011}

%\subjclass{34E15, 34C37, 37G99}

%\keywords{geometric singular perturbation theory, blow-up, pitchfork
%bifurcation, transcritical bifurcation}

%\thanks{The research of S. S.  was supported in part by the National Science Foundation
%under grant DMS-0708386.}
%\thanks{The research of C. S. was supported in part by FONDECYT under grant 3085026.}
\begin{abstract}
In this paper we provide a formal matched asymptotic analysis for large solutions to the Gelfand-Liouville problem in planar, doubly connected domains in the plane. Using these, we rigorously construct a good approximate solution to the problem.
\end{abstract}

\maketitle \tableofcontents
\section{Introduction}
\subsection{The problem}
We consider the Gelfand-Liouville problem:
\begin{equation}\label{eqEqradialmain}
\left\{\begin{array}{ll}
  \Delta u+\lambda^2 e^u=0 & \textrm{in}\ \Omega, \\
   &  \\
  u=0 &\textrm{on}\ \partial\Omega,
\end{array}\right.
\end{equation}
where $\Omega$ is a bounded, smooth domain of $\mathbb{R}^N$,
$N\geq 2$, and $\lambda>0$ is a parameter.

\subsection{Motivation}
\subsubsection{Motivation from physics} This equation arises in the
theories of thermionic emission (see \cite{gelfand}), isothermal
gas spheres (see \cite{chadreskar}), and gas combustion (see
\cite{mignot}). Furthermore, this type of equations appear in
statistical mechanics (see \cite{stat1}, \cite{stat2},
\cite{statiChan}).

\subsubsection{Motivation from geometry} This type of equations
arise in the prescribed Gaussian and scalar curvature problems in
a compact manifold (see \cite{chang1}, \cite{chang2},
\cite{kazdan}).

\subsection{Know results} Problem (\ref{eqEqradialmain}) has been studied extensively. With no hope of being
complete, let us mention the following results which are more
related to the scope of the current paper.

\subsubsection{General $N\geq 2$} It is well known  that there
exists a $\lambda_*>0$ such that (\ref{eqEqradialmain}) admits a
minimal
 solution $\underline{u}_\lambda$ if $\lambda \in (0, \lambda_*)$ (here minimal means
 smallest);
no  solution if $\lambda> \lambda_*$; admits a solution if
$\lambda=\lambda_*$ and $N\leq 9$ (see \cite{crandal},
\cite{davila}, \cite{dupaigne} and \cite{mignot}). In fact, the
minimal solution $\underline{u}_\lambda$ can be constructed by the
method of upper and lower solutions and satisfies
\[
\|\underline{u}_\lambda \|_{L^\infty(\Omega)}\to 0\ \ \textrm{as}\
\ \lambda \to 0.
\]

The simply connectedness of the domain plays an important role in
the structure of solutions to (\ref{eqEqradialmain}). This can
already be made clear by looking at the case of radially symmetric
domains. In the case of a ball, a classical result of Gidas, Ni
and Nirenberg \cite{gidas} implies that every solution is radially
symmetric and decreasing (it goes without saying that every
solution in any domain is positive). Therefore, the problem
reduces to an ordinary differential equation. Based on this
observation, Joseph and Lundgren \cite{josephLudgren} were able to
completely characterize the
 structure of solutions in all  dimensions (see also \cite{gelfand} and \cite{mignatBall}). Of
particular interest is the relationship they observed between the
multiplicity of solutions and the space dimension. On the other
hand, in the case of an annulus, there exists a continuous curve
of radial solutions along which infinitely many  bifurcations to
non-radial solutions takes place (see \cite{lin1}, \cite{lin2},
\cite{pacardAnn}). The aforementioned radial solutions are
critical points of mountain-pass type  for the associated energy
in the natural energy space of radially symmetric functions (see
\cite{Grossi-asymptotic analysis}). In fact, if $\lambda>0$ is
sufficiently small,  problem (\ref{eqEqradialmain}) has exactly
two radial solutions, namely the minimal one and the mountain-pass
in the class of radial solutions (see \cite{lin1} and
\cite{nagasaki-suzuki}). In particular, when $N=2$, these
solutions can be given explicitly by using their invariance
through a transformation group (see \cite{gelfand}, \cite{lin1}).
\subsubsection{Behaviour as $\lambda \to 0^+$}
\textbf{\underline{$N=2$.}} In this case, thanks to the works
\cite{brezisMerle}, \cite{lishaffrir}, \cite{nagasaki-suzuki} and
\cite{suzukiProc}, we can classify all possible solutions to
(\ref{eqEqradialmain}) by the limit of the quantity
\[
\mathcal{T}_\lambda=\lambda^2\int_{\Omega}^{}e^udx.
\]
Loosely speaking, if $T_\lambda$ remains bounded as $\lambda \to 0^+$ then the solutions blow-up at a finite number of points in $\Omega$. More precisely, they exhibit ''bubbling" behaviour at a fixed finite number of points in $\Omega$: after a proper rescaling near each such point, the solution resembles the unique solution of the following limit problem:
\[
\Delta u+e^u=0\ \ \textrm{in}\ \ \mathbb{R}^2,\ \ \int_{\mathbb{R}^2}e^udx<\infty.
\]
\textbf{\underline{$N\geq 2$.}} In any dimension, the asymptotic behaviour of radial mountain-pass solutions in an annulus was investigated
in \cite{Grossi-asymptotic analysis}. It was shown there in that $\mathcal{T}_\lambda$ diverges to infinity as $\lambda \to 0$. Moreover, the solutions blow-up in the whole of $\Omega$ and in fact concentrate their energy around a special hyper-sphere. After a proper rescaling, the solutions resemble that of
the unique one-dimensional bubble
\begin{equation}\label{eq1d}
 u''+e^u=0\ \ \textrm{in}\ \ \mathbb{R},\ \ \int_{\mathbb{R}}e^udx<\infty,
\end{equation}
transplanted along that hyper-sphere.

\subsection{The problem and the main result}
It would be of great interest to see if there exist in general domains  analogous solutions to the radial ones of \cite{Grossi-asymptotic analysis}.
In this paper we make a modest step in this direction by carrying out successfully matched asymptotic expansions for the problem in the simplest possible nonradial situation: In the case where $\Omega$ is two-dimensional doubly connected domain in the plane. We would like to point out that, to the best of our knowledge, our calculations are new even in the radial case (there they carry out considerably more easily and in any dimension).

In our main result, stated in Proposition \ref{proRemainderGlobal} below, we use these asymptotic expansions to construct a good global approximate solution to (\ref{eqEqradialmain}) for small $\lambda>0$, which blows up in the whole of $Omega$ and concentrates its energy along a special curve in $\Omega$.
Near that curve, and in the normal direction to it, that approximate solution given to main order by a properly rescaled solution of (\ref{eq1d}). Moreover, using our detailed estimates, one can also calculate an asymptotic expansion for $\mathcal{T}_\lambda$ as $\lambda \to 0^+$.

It is natural to believe that there exists a genuine solution of (\ref{eqEqradialmain}) near the approximate one.
In the radial case, this can be made rigorous by a linearization argument in some carefully chosen weighted spaces and expanding on some ideas from
\cite{AS}. However, in the nonradial setting at hand there are various issues which prevent a straightforward application of the techniques in the aforementioned reference. Further, in light of our previous discussion on nonradial bifurcation, a new difficulty is also expected to occur by the presence of subtle resonance phenomena (see also \cite{delPistoia}).

\section{The curve $\gamma$ and its harmonic measures $W_\gamma^\pm$}
\begin{proposition}\label{proW}
There exists a smooth closed Jordan curve $\gamma$ in $\Omega$,
dividing $\Omega$ in two open domains $\Omega^+$ (the outer) and
$\Omega^-$ (the inner), with the following property:

Suppose that $u^\pm$ satisfy classically
\begin{equation}\label{eqVpm}
\left\{\begin{array}{ll}
  \Delta u^\pm=0 &\textrm{in}\ \ \Omega^\pm,  \\
    &   \\
   u^\pm =0& \textrm{on}\ \ \partial\Omega^\pm\cap \partial \Omega,
\end{array}\right.
\end{equation}
and
\begin{equation}\label{eqkfermi}
\partial_t^ku^+=(-1)^k\partial_t^ku^-\ \ \textrm{on}\ \ \gamma,
\end{equation}
for some integer $k\in \{0,1 \}$; where $\partial_t$ denotes
derivation in the direction perpendicular to the curve $\gamma$
(see (\ref{eqFermi}) below). Then, it holds that
\begin{equation}\label{eq22+}
\partial_t^{1-k}u^+=(-1)^{1-k}\partial_t^{1-k}u^-\ \ \textrm{on}\ \
\gamma.
\end{equation}
\end{proposition}
 \begin{proof}
Since $\Omega$ is doubly connected, by Theorem 5.10h in
\cite{henrici} (see also \cite[Chpt. 17]{henrici3}), there exists
a one-to-one transformation $w=f(z)$ that maps $\Omega$
conformally onto the annulus \[\mathcal{A}=\{a<r<b\},\] where
$r=|w|$, for some $a,b>0$. Here we have identified $\mathbb{R}^2$
with the complex plane. In fact, the ratio $b/a$ is unique.
Furthermore, without loss of generality, we may assume that
$\partial \Omega^+$ maps onto $r=b$. We will first verify the
assertions of the proposition for the case of the annulus,
exploiting that the corresponding solutions can be represented
explicitly, and then argue that they continue to hold for the
general case by transplanting back.

  In the ``ideal'' case where $\Omega=\mathcal{A}$ and
$\gamma$ is the circle \[\mathcal{C}=\{r=\sqrt{ab}\}\] (this is
not a guess but is based on the analysis in
\cite{Grossi-asymptotic analysis}), we have the following explicit
formulas for the solutions $V^\pm$ of (\ref{eqVpm}) (see
\cite[Chpt. 15]{henrici3} or \cite[Chpt. 6]{strauss}):
\[
V^-(r,\theta)=A_0^-
\ln\left(\frac{r}{a}\right)+\sum_{n=1}^{\infty}A_n^-\left[\left(\frac{r}{a}
\right)^n-\left(\frac{r}{a} \right)^{-n} \right]\cos(n\theta)+
\sum_{n=1}^{\infty}B_n^-\left[\left(\frac{r}{a}
\right)^n-\left(\frac{r}{a} \right)^{-n} \right]\sin(n\theta),
\]
$a\leq r \leq\sqrt{ab}$, $-\pi\leq\theta\leq \pi$, and
\[
V^+(r,\theta)=A_0^+
\ln\left(\frac{b}{r}\right)+\sum_{n=1}^{\infty}A_n^+\left[\left(\frac{r}{b}
\right)^n-\left(\frac{r}{b} \right)^{-n} \right]\cos(n\theta)+
\sum_{n=1}^{\infty}B_n^+\left[\left(\frac{r}{b}
\right)^n-\left(\frac{r}{b} \right)^{-n} \right]\sin(n\theta),
\]
$\sqrt{ab}\leq r\leq b$, $-\pi\leq\theta\leq \pi$, where $A_0^\pm$
and $A_n^\pm, B_n^\pm$, $n=1,\cdots$ are arbitrary constants. In
particular, we find that
\[
V^-(\sqrt{ab},\theta)=A_0^-
\ln\sqrt{\frac{b}{a}}+\sum_{n=1}^{\infty}A_n^-\left[\left(\frac{b}{a}
\right)^{\frac{n}{2}}-\left(\frac{b}{a} \right)^{-\frac{n}{2}}
\right]\cos(n\theta)+
\sum_{n=1}^{\infty}B_n^-\left[\left(\frac{b}{a}
\right)^{\frac{n}{2}}-\left(\frac{b}{a} \right)^{-\frac{n}{2}}
\right]\sin(n\theta),
\]
\[
V^+(\sqrt{ab},\theta)=A_0^+
\ln\sqrt{\frac{b}{a}}+\sum_{n=1}^{\infty}A_n^+\left[\left(\frac{a}{b}
\right)^\frac{n}{2}-\left(\frac{a}{b} \right)^{-\frac{n}{2}}
\right]\cos(n\theta)+
\sum_{n=1}^{\infty}B_n^+\left[\left(\frac{a}{b}
\right)^\frac{n}{2}-\left(\frac{a}{b} \right)^{-\frac{n}{2}}
\right]\sin(n\theta),
\]
and
\[
V_r^-(\sqrt{ab},\theta)=\frac{A_0^-}{\sqrt{ab}}
+\sum_{n=1}^{\infty}\frac{nA_n^-}{a}\left[\left(\frac{b}{a}
\right)^{\frac{n-1}{2}}+\left(\frac{b}{a} \right)^{-\frac{n+1}{2}}
\right]\cos(n\theta)+
\sum_{n=1}^{\infty}\frac{nB_n^-}{a}\left[\left(\frac{b}{a}
\right)^{\frac{n-1}{2}}+\left(\frac{b}{a} \right)^{-\frac{n+1}{2}}
\right]\sin(n\theta),
\]
\[
V_r^+(\sqrt{ab},\theta)=-\frac{A_0^+}{\sqrt{ab}}
+\sum_{n=1}^{\infty}\frac{nA_n^+}{b}\left[\left(\frac{a}{b}
\right)^\frac{n-1}{2}+\left(\frac{a}{b} \right)^{-\frac{n+1}{2}}
\right]\cos(n\theta)+
\sum_{n=1}^{\infty}\frac{nB_n^+}{b}\left[\left(\frac{a}{b}
\right)^\frac{n-1}{2}+\left(\frac{a}{b} \right)^{-\frac{n+1}{2}}
\right]\sin(n\theta),
\]
for $-\pi\leq\theta\leq \pi$. If $V^\pm$ satisfy (\ref{eqkfermi})
for some $k\in \{0,1\}$ (with $\gamma=\mathcal{C}$), it follows
that $A_0^-=A_0^+$, and $A_n^-=-A_n^+$, $B_n^-=-B_n^+$,
$n=1,\cdots$ (we remark that $\partial_t=-\partial_r$). It is then
straightforward to verify that this implies that (\ref{eq22+})
holds.

Let us assume now that $\Omega\subset \mathbb{R}^2$ is a general
 bounded, doubly connected domain with smooth boundary and that $f$ is a
one-to-one conformal map that maps $\Omega$ onto the annulus
$\mathcal{A}$, as described above. Set
\[
\gamma=f^{-1}\left( \mathcal{C}\right),
\]
and define $\Omega^\pm\subset \Omega$ accordingly. Assume that
$u^\pm$ satisfy (\ref{eqVpm}) and let
\begin{equation}\label{eqcon2}
V^\pm(w)=u^\pm \left(f^{-1}(w) \right),\ \ w\in \mathcal{A},
\end{equation}
denote their corresponding conformal transplants under the mapping
$f$. As is well known, the functions $V^\pm$ are still harmonic
and thus satisfy (\ref{eqVpm}) with $\Omega=\mathcal{A}$ and
$\gamma=\mathcal{C}$.
 If
\[
u^+=u^-\ \ \textrm{on}\ \ \gamma,
\]
then clearly
\[
V^+=V^-\ \ \textrm{on}\ \ \mathcal{C}.
\]
As we showed previously, the above relation implies that
\begin{equation}\label{eqcon1}
\left(\nabla V^-(w),\nu_w \right)+\left(\nabla V^+(w),\nu_w
\right)=0\ \ \textrm{on}\ \mathcal{C},
\end{equation}
where $\nu_w$ stands for the inner unit normal to $\mathcal{C}$ at
$w\in \mathcal{C}$, while $(\cdot,\cdot)$ stands for the Euclidean
inner product in $\mathbb{R}^2$. If
\[
z=f^{-1}(w)\in \gamma,
\]
and $n_z$ denotes the inner unit normal to $\gamma$ at $z$, we
have that
\begin{equation}\label{eqlambdanormal}
\nu_w=\kappa Df(z)n_z\ \ \textrm{for some}\ \kappa\in \{-1,1 \}.
\end{equation} In order not to interrupt the line of thought, we will show this at
the end of the proof. Therefore, in view of (\ref{eqcon2}) and
(\ref{eqcon1}), we obtain that
\[
\left(Df^{-1}(w)\nabla u^-(z),Df(z)n_z
\right)+\left(Df^{-1}(w)\nabla u^+(z),Df(z)n_z \right)=0\ \
\textrm{on}\ \gamma.
\]
Using that $Df^{-1}(w)Df(z)=I$ and the orthogonality of
$Df^{-1}(w)$ (by the fact that $f^{-1}$ is conformal), we deduce
that
\[
\left(\nabla u^-(z),n_z \right)+\left(\nabla u^+(z),n_z \right)=0\
\ \textrm{on}\ \gamma,
\]
as desired.

If $\partial_t u^+=-\partial_t u^-$ on $\gamma$, we can show that
$u^+=u^-$ on $\gamma$ by reversing the above steps.

 It remains to show (\ref{eqlambdanormal}). To
this end, let $\tau_z$ be a (nontrivial) tangent vector to
$\gamma$ at $z$. Then, since $f(\gamma)=\mathcal{C}$, it follows
that the vector
\[
T_w=Df(z)\tau_z
\]
is tangent to $\mathcal{C}$ at $w$. Moreover, since the matrix
$Df(z)$ is orthogonal (by the conformality of $f$), that is
$Df(z)Df(z)=I$, we see that $|T_w|=|\tau_z|$. The orthogonality of
$Df(z)$ implies in addition that the angle between $\tau_z$ and
$n_z$ is the same as the one between $T_w$ and $Df(z)n_z$. Hence,
the vector $Df(z)n_z$ is normal to $\mathcal{C}$ at $w$. So, since
it has unit norm (again by the orthogonality of $Df(z)$), we infer
that (\ref{eqlambdanormal}) holds.
 \end{proof}
\begin{definition}
\cite[Chpt. 15.5]{henrici3} The unique functions $W_\gamma^+,\
W_\gamma^-$ such that
\begin{equation}\label{eqHar-measure}
\left\{\begin{array}{ll}
  \Delta W_\gamma^\pm=0 &\textrm{in}\ \ \Omega^\pm,  \\
    &   \\
   W_\gamma^\pm =0& \textrm{on}\ \ \partial\Omega^\pm\cap \partial
   \Omega,\\
&\\
W_\gamma^\pm=1&\textrm{on}\ \ \gamma,
\end{array}\right.
\end{equation}
are called the \emph{harmonic measures} of $\gamma$ with respect
to $\partial\Omega^+\cap \partial \Omega$, $\partial\Omega^-\cap
\partial \Omega$ respectively.
\end{definition}
It follows from Proposition \ref{proW} that
\begin{equation}\label{eqHarmonic-measure-derivs}
\partial_nW_\gamma^++\partial_nW_\gamma^-=0\ \ \textrm{on}\ \
\gamma,
\end{equation}
where $n$ denotes the unit normal vector to $\gamma$ pointing in
$\Omega^-$. Furthermore, by the strong maximum principle, we
deduce that
\begin{equation}\label{eqW<1}
0<W_\gamma^\pm<1\ \ \textrm{in}\ \ \Omega^\pm. \end{equation}
Moreover, from Hopf's boundary point lemma, we have that
\begin{equation}\label{eqHopf}
\pm \partial_n W_\gamma^\pm>0\ \ \textrm{on}\ \ \gamma.
\end{equation}
\section{The function $V_0$ and the associated linearized problem}
The function
\begin{equation}\label{eqUU0def}
V_0(x)=\ln \frac{4e^{\sqrt{2}x}}{\left(1+e^{\sqrt{2}x} \right)^2}
\end{equation}
solves
\begin{equation}\label{eqUU}
  v_{xx}+e^{v}=0, \ \ \ x\in \mathbb{R}.
\end{equation}
We see that $V_0(0)=(V_0)_x(0)=0$, $V_0$ is even and
\begin{equation}\label{eqasymptoticU0U0s}
\left\{\begin{array}{lcll}
  V_0(x) & = & \sqrt{2}x+\ln 4+\mathcal{O}\left(e^{-\sqrt{2}|x|}\right), & x\to -\infty, \\
   &  &  &  \\
  (V_0)_x(x) & = & \sqrt{2}+\mathcal{O}\left(e^{-\sqrt{2}|x|}\right), & x\to
  -\infty.
\end{array}\right.
\end{equation}
Note that
\begin{equation}\label{eqUdilation} V(x)=2\ln
\mu+V_0\left(\mu(x-h)\right)
\end{equation}
 also solves (\ref{eqUU}) for every
$\mu>0$ and $h \in \mathbb{R}$. It follows that the linear
equation
\begin{equation}\label{eqlinearhomogenous}
\psi_{xx}+e^{V_0}\psi=0,\ \ x\in \mathbb{R},
\end{equation}
has two linearly independent solutions given by
\[
(V_0)_x\ \ \ \textrm{and}\ \ \ x(V_0)_x+2.
\]
The variation of constants formula yields (see also
\cite{grossi-jfa}):
\begin{lemma}\label{lemlinearvariation}
Suppose that $g\in C(\mathbb{R})$ satisfies $|g(x)|+|g_x(x)|\leq
De^{-d|x|},\ x\in \mathbb{R}$, for some constants $d,D>0$.

If $g$ is even, then the linear equation
\begin{equation}\label{eqlinearinhom}
\psi_{xx}+e^{V_0}\psi=g,\ \ \ x\in \mathbb{R},
\end{equation}
has a one parameter family of even solutions  satisfying
\[\begin{array}{ll}
    \psi(x)=-\sqrt{2}\Delta
x+\frac{1}{\sqrt{2}}\int_{-\infty}^{0}\left(t(V_0)_t(t)+2\right)g(t)dt-2\Delta+\mathcal{O}\left(
e^{-\alpha|x|}\right)+\Delta \mathcal{O}\left(
e^{-\sqrt{2}|x|}\right),& x\to -\infty,
     \\
    \psi_x(x)=-\sqrt{2}\Delta
+\mathcal{O}\left( e^{-\alpha|x|}\right)+\Delta \mathcal{O}\left(
e^{-\sqrt{2}|x|}\right),& x\to -\infty,
  \end{array}
\]
 for every $\Delta \in \mathbb{R}$, where $\alpha=\alpha(d,D)>0$.

If $g$ is odd, then (\ref{eqlinearinhom}) has a one parameter
family of odd solutions satisfying
\[\begin{array}{ll}
  \psi(x)=-\frac{x}{\sqrt{2}}\int_{-\infty}^{0}(V_0)_t(t)g(t)dt-E+\mathcal{O}\left(
e^{-\alpha|x|}\right)+E\mathcal{O}\left( e^{-\sqrt{2}|x|}\right), & x\to -\infty, \\
   &  \\
  \psi_x(x)=-\frac{1}{\sqrt{2}}\int_{-\infty}^{0}(V_0)_t(t)g(t)dt+\mathcal{O}\left(
e^{-\alpha|x|}\right)+E\mathcal{O}\left( e^{-\sqrt{2}|x|}\right),
& x\to -\infty,
\end{array}
\]
for every $E \in \mathbb{R}$, where $\alpha=\alpha(d,D)>0$.
\end{lemma}
%\section{Notation}By $u_i(r)$ we will denote the inner solution
%defined close to $r=r_0$, and by $v_i(r)$ the outer solution defined
%away from $r=r_0$ at the $i$-step, $i=0,1,\cdots$. By $c,C>0$ we
%will denote generic constants independent of $\lambda$ as well as
%$\mu, M, K, h, \Gamma_{i\pm}, \Delta_i, E_i$ introduced in the
%following sections. By $\mathcal{O}$ we will denote the usual Landau
%symbol , i.e., $|\mathcal {O}(t)|\leq C |t|$ for every $t$.

\section{The inner approximation}
\subsection{The set up near the curve $\gamma$} Let $\gamma$ be the closed smooth curve in Proposition
\ref{proW}, and $ \ell= |\Gamma|$ its total length. We consider
the natural parametrization $\gamma=\gamma (s)$ of $\Gamma$ with
positive orientation, where $s$ denotes an arc length parameter
measured from a fixed point of $\Gamma$. Let $n(s)$ denote the
inner unit normal to $\Gamma$. Points $y$ that are
$\delta_0$-close to $\Gamma$, for sufficiently small $\delta_0$,
can be represented in the form
\begin{equation}\label{eqFermi}y=\gamma(s)+tn(s),\ \  \ s\in [0,\ell),\ \ |t | < \delta_0,
\end{equation}
 where the map
$y\mapsto (s, t)$ is a local diffeomorphism. Note that we have
$0<t<\delta_0$ in $\Omega^-$.

For any smooth function $u$ that is defined in this region,
identifying $u(y)$ with $u(s,t)$ (allowing some abuse of
notation), we have that
\begin{equation}\label{eqgradFermi}
\nabla_y u=\left(\frac{u_s}{1-kt},u_t\right),
\end{equation}
(with the obvious interpretation).
Therefore, equation (\ref{eqEqradialmain}) for $u$ expressed in
these coordinates becomes
\begin{equation}\label{eqS(u)}
S(u):=u_{tt}+\frac{1}{a}u_{ss}-\frac{\partial_s
a}{2a^2}u_s+\frac{\partial_t a}{2a}u_t+\lambda^2 e^u=0,
\end{equation}
in the region described in (\ref{eqS(u)}), where
$a=\left(1-tk(s)\right)^2$, and $k$ is the curvature of $\gamma$.

Let $\mu, f\in C_{per}^2([0,\ell])$, the space of $\ell$-periodic,
$C^2$-functions, with $\mu>0$, and
\begin{equation}\label{eqmu,f-norms}
c\ln \frac{1}{\lambda}\leq\lambda\mu\leq C\ln \frac{1}{\lambda},\
\ \lambda\mu f=o\left(\ln(\lambda\mu)\right)\ \  \textrm{as}\
\lambda\to 0\ \ \textrm{on} \ \gamma.
\end{equation}
 Given a smooth function $u$,
defined close to the curve $\gamma$, let
\begin{equation}\label{eqv}
u(s,t)=v(s,x)+2\ln \mu,\ \ \ x=\lambda \mu (t-f),\ s\in [0,\ell],\
|t|<\delta_0.
\end{equation}
We want to express equation (\ref{eqEqradialmain}) in terms of
these new coordinates (in the region described in (\ref{eqv})). We
compute:
\begin{equation}\label{equt-us-uss}
\left\{\begin{array}{lll}
 u_t & = & \lambda\mu v_x, \\
    &   &   \\
  u_{tt} & = & \lambda^2\mu^2v_{xx}, \\
    &   &   \\
  u_s & = & v_s+(\mu'\mu^{-1}x-\lambda \mu f')v_x+2\mu'\mu^{-1}, \\
    &   &   \\
  u_{ss} & =  & v_{ss}+2(\mu'\mu^{-1}x-\lambda\mu f')v_{sx}+(\mu'\mu^{-1}x-\lambda \mu f')^2v_{xx} \\
    &   & +(\mu''\mu^{-1}x-2\lambda \mu'f'-\lambda\mu
    f'')v_x+2\mu''\mu^{-1}-2(\mu')^2\mu^{-2}.
\end{array}\right.
\end{equation}
A short calculation shows that $u$ solves (\ref{eqS(u)}) if and
only if $v$, defined in (\ref{eqv}), solves $S(v+2\ln \mu)=0$,
where
\begin{equation}\label{eqS(v+2ln)}
\begin{array}{lll}
  S(v+2\ln \mu) & = & \lambda^2\mu^2v_{xx}+\frac{1}{[1-(\lambda^{-1}\mu^{-1}x+f)k]^2}\left[
  v_{ss}+2(\mu'\mu^{-1}x-\lambda\mu f')v_{sx}+(\mu'\mu^{-1}x-\lambda \mu f')^2v_{xx} \right. \\
    &   &   \\
    &   &\left.+(\mu''\mu^{-1}x-2\lambda \mu'f'-\lambda\mu
    f'')v_x+2\mu''\mu^{-1}-2(\mu')^2\mu^{-2}\right]-\frac{k}{1-(\lambda^{-1}\mu^{-1}x+f)k}\lambda\mu v_x   \\
    &   &   \\
    &
    &+\frac{(\lambda^{-1}\mu^{-1}x+f)k'}{[1-(\lambda^{-1}\mu^{-1}x+f)k]^3}\left[v_s+(\mu'\mu^{-1}x-\lambda \mu f')v_x+2\mu'\mu^{-1}
    \right]+\lambda^2\mu^2 e^v.
\end{array}
\end{equation}
\subsection{The first order inner approximate solution} As a first
order approximation, valid near the curve $\gamma$, we consider
$u_0$ as described in (\ref{eqv}) with $v=V_0(x)$, defined in
(\ref{eqUU0def}), i.e.,
\begin{equation}\label{equ0}
u_0(s,t)=2\ln \mu+ V_0\left(\lambda\mu(t-f)\right),\ \ s\in
[0,\ell],\ \ |t|\leq\lambda^{-1}\mu^{-1}L.
\end{equation}
Here $L$ satisfies
\begin{equation}\label{eqL}
M\ln\left(\ln\frac{1}{\lambda}\right)\leq L\leq 2M
\ln\left(\ln\frac{1}{\lambda}\right),
\end{equation}
for some large constant $M>0$ to be determined independently of
$\lambda$ (recall (\ref{eqmu,f-norms})).

In view of (\ref{eqS(v+2ln)}), we have
\begin{equation}\label{eqS(u0)}
\begin{array}{lll}
  S(u_0) & = & \frac{1}{[1-(\lambda^{-1}\mu^{-1}x+f)k]^2}\left[
  -(\mu'\mu^{-1}x-\lambda\mu f')^2e^{V_0}+(\mu''\mu^{-1}x-2\lambda \mu'f'-\lambda\mu
    f'')(V_0)_x\right. \\
    &   &   \\
    &   &\left.+2\mu''\mu^{-1}-2(\mu')^2\mu^{-2}\right]-k\lambda\mu (V_0)_x-k\left[
    \frac{1}{1-(\lambda^{-1}\mu^{-1}x+f)k}-1\right]\lambda\mu (V_0)_x   \\
    &   &   \\
    &
    &+\frac{(\lambda^{-1}\mu^{-1}x+f)k'}{[1-(\lambda^{-1}\mu^{-1}x+f)k]^3}\left[(\mu'\mu^{-1}x-\lambda \mu f')(V_0)_x+2\mu'\mu^{-1}
    \right].
\end{array}
\end{equation}
\subsection{The second order inner approximate solution}
We search a refined approximate solution, valid near the curve
$\gamma$, in the form
\begin{equation}\label{equ1}
u_1(s,t)=2\ln \mu+
V_0\left(\lambda\mu(t-f)\right)+\phi\left(s,\lambda\mu(t-f)\right),\
\ s\in [0,\ell],\ \ |t|\leq\lambda^{-1}\mu^{-1}L,
\end{equation}
with $\phi$ to be determined.

In view of (\ref{eqS(v+2ln)}), (\ref{eqS(u0)}), we have
\begin{equation}\label{eqS(u1)}
\begin{array}{lll}
  S(u_1) & = & S(u_0)+\lambda^2\mu^2\left[e^{V_0+\phi}-e^{V_0}-e^{V_0}\phi \right]
+\lambda^2\mu^2\left[\phi_{xx}+e^{V_0}\phi \right]\\
&&\\ & &+
  \frac{1}{[1-(\lambda^{-1}\mu^{-1}x+f)k]^2}\left[\phi_{ss}+2
  (\mu'\mu^{-1}x-\lambda\mu f')\phi_{sx}+(\mu'\mu^{-1}x-\lambda\mu f')^2\phi_{xx}\right. \\
    &   &   \\
    &   &\left.+(\mu''\mu^{-1}x-2\lambda \mu'f'-\lambda\mu
    f'')\phi_x\right]-\frac{k}{1-(\lambda^{-1}\mu^{-1}x+f)k}\lambda\mu \phi_x   \\
    &   &   \\
    &
    &+\frac{(\lambda^{-1}\mu^{-1}x+f)k'}{[1-(\lambda^{-1}\mu^{-1}x+f)k]^3}\left[\phi_s+(\mu'\mu^{-1}x-\lambda \mu f')\phi_x\right].
\end{array}
\end{equation}
Since, at least formally,
\[
S(u_0)=-k\lambda\mu (V_0)_x+\textrm{lower\ order\ terms},
\]
we choose $\phi=\phi_1(s,x)$ to satisfy, for fixed $s\in
[0,\ell]$,
\begin{equation}\label{eqphi(1)}
\phi_{xx}+e^{V_0}\phi=\lambda^{-1}\mu^{-1}k(V_0)_{x},\ \ \ x\in
\mathbb{R}.
\end{equation}

\begin{lemma}\label{lemphi1}
Equation (\ref{eqphi(1)}) has a family of solutions such that
\[\begin{array}{llll}\lambda\mu \phi_1(s,x)& = &
\frac{1}{\sqrt{2}}kx^2+2kx-\sqrt{2}\Delta_1x-2\Delta_1-E_1+(1+|\Delta_1|+|E_1|)\mathcal{O}\left(e^{-c|x|}
\right), & x\to  -\infty,
   \\
   &  &  &  \\
  \lambda \mu (\phi_1)_x(s,x) & = &  \sqrt{2}kx+2k-\sqrt{2}\Delta_1+(1+|\Delta_1|+|E_1|)\mathcal{O}\left(e^{-c|x|} \right), & x\to
  -\infty,
\\
&  &   & \\
 &  &   &   \\
\lambda\mu \phi_1(s,x) &=  & -\frac{1}{\sqrt{2}}
kx^2+2kx+\sqrt{2}\Delta_1x-2\Delta_1+E_1+(1+|\Delta_1|+|E_1|)\mathcal{O}\left(e^{-c|x|} \right), & x\to \infty, \\
   &    &  &\\
\lambda \mu (\phi_1)_x(s,x) &=  &
-\sqrt{2}kx+2k+\sqrt{2}\Delta_1+(1+|\Delta_1|+|E_1|)\mathcal{O}\left(e^{-c|x|}
\right), & x\to \infty,
 \end{array}
\]
for every $\Delta_1,\ E_1 \in C_{per}^2([0,\ell])$.
\end{lemma}
\begin{proof}
Let
\begin{equation}\label{eqZ1def}
Z_1(s,x)=k\frac{x^2}{2}(V_0)_x,\ \ \ x\in \mathbb{R}.
\end{equation}
We see that
\[
\begin{array}{rcl}
  (Z_1)_{xx}+e^{V_0}Z_1
 & = & k(V_0)_x+k\frac{x^2}{2}(V_0)_{xxx}+2kx(V_0)_{xx}+k\frac{x^2}{2}e^{V_0}(V_0)_x \\
   &  &  \\
   &= & k(V_0)_x-2kxe^{V_0}.
\end{array}
\]
We write
\begin{equation}\label{eqWZ+fi}
\phi=\lambda^{-1}\mu^{-1}Z_1+\psi.
\end{equation}
In terms of $\psi$, equation (\ref{eqphi(1)}) becomes
\begin{equation}\label{eqpsi}
\psi_{xx}+e^{V_0}\psi=2\lambda^{-1}\mu^{-1}kxe^{V_0}.
\end{equation}
We can now apply Lemma \ref{lemlinearvariation} to the above
equation, since its righthand side decays exponentially to zero as
$x\to \pm \infty$ and is odd (plus the trivial even function).
Note that
\[
\int_{-\infty}^{0}xe^{V_0}(V_0)_xdx=-\int_{-\infty}^{0}e^{V_0}dx=\int_{-\infty}^{0}(V_0)_{xx}dx=-\sqrt{2}.
\]
Hence, for every $\Delta_1,\ E_1\in C_{per}^2([0,\ell])$, there
exists a solution $\psi_1$ of (\ref{eqpsi}) such that
\[
\begin{array}{ll}
  \lambda\mu\psi_1(s,x)=2kx+\sqrt{2}\Delta_1x-2\Delta_1+E_1+(1+|\Delta_1|+|E_1|)\mathcal{O}\left(e^{-c|x|}\right), & x\to \infty, \\
   &  \\
  \lambda\mu\psi_1(s,x)=2kx-\sqrt{2}\Delta_1x-2\Delta_1-E_1+(1+|\Delta_1|+|E_1|)\mathcal{O}\left(e^{-c|x|}\right), & x\to
  -\infty,
\end{array}
\]
and the corresponding estimates for $(\psi_1)_x$.

Let $\phi_1$ be defined by relation (\ref{eqWZ+fi}) with
$\psi=\psi_1$. Then $\phi_1$ solves equation (\ref{eqphi(1)}) and
has the desired asymptotic behavior as in the assertion of the
lemma.

The proof of the lemma is complete.
\end{proof}

In view of the above lemma, and (\ref{equ1}), we have
\begin{equation}\label{equ1-}
\begin{array}{lll}
  u_1(s,-\lambda^{-1}\mu^{-1}L) & = & 2\ln(2\mu)-\sqrt{2}(L+\lambda\mu f)+\lambda^{-1}\mu^{-1}\frac{k}{\sqrt{2}}(L+\lambda\mu
  f)^2-
  2\lambda^{-1}\mu^{-1}k(L+\lambda\mu f) \\
    &   &   \\
    &   & +\sqrt{2}\Delta_1 \lambda^{-1}\mu^{-1}(L+\lambda\mu f)-2\lambda^{-1}\mu^{-1}\Delta_1-\lambda^{-1}\mu^{-1}E_1 \\
    && \\
    &&+\lambda^{-1}\mu^{-1}(1+|\Delta_1|+|E_1|)e^{-c(L+\lambda\mu
    f)},
\end{array}
\end{equation}
and
\begin{equation}\label{equ1+}
\begin{array}{lll}
  u_1(s,\lambda^{-1}\mu^{-1}L) & = & 2\ln(2\mu)-\sqrt{2}(L-\lambda\mu f)-\lambda^{-1}\mu^{-1}\frac{k}{\sqrt{2}}(L-\lambda\mu
  f)^2+
  2\lambda^{-1}\mu^{-1}k(L-\lambda\mu f) \\
    &   &   \\
    &   & +\sqrt{2}\Delta_1 \lambda^{-1}\mu^{-1}(L-\lambda\mu f)-2\lambda^{-1}\mu^{-1}\Delta_1+\lambda^{-1}\mu^{-1}E_1 \\
    && \\
    &&+\lambda^{-1}\mu^{-1}(1+|\Delta_1|+|E_1|)e^{-c(L-\lambda\mu
    f)}.
\end{array}
\end{equation}
\subsection{The third order inner approximate solution}
We search a refined approximate solution, valid near the curve
$\gamma$, in the form
\begin{equation}\label{equ2}
\begin{array}{lll}
  u_2(s,t) & = & u_1(s,t)+\phi\left(s,\lambda\mu(t-f)\right) \\
    &  &   \\
    &  = & 2\ln \mu+
V_0\left(\lambda\mu(t-f)\right)+\phi_1\left(s,\lambda\mu(t-f)\right)+\phi\left(s,\lambda\mu(t-f)\right),
\end{array}
\end{equation}
$s\in [0,\ell],\ \ |t|\leq\lambda^{-1}\mu^{-1}L$, with $\phi$ to
be determined.

In view of (\ref{eqS(v+2ln)}), by a careful calculation and
rearrangement of terms, we can write
\begin{equation}\label{eqS(u2)}
\begin{array}{lll}
  S(u_2) & = & \lambda^2\mu^2(\phi_{xx}+e^{V_0}\phi)+\frac{1}{\left[
1-(\lambda^{-1}\mu^{-1}x+f)k\right]^2}\left[(\phi_1+\phi)_{ss}+2(\mu'\mu^{-1}x-\lambda\mu
f')(\phi_1+\phi)_{sx} \right.\\
&&\\
& &+(\mu'\mu^{-1}x-\lambda\mu
f')^2(\underbrace{V_0}+\phi_1+\phi)_{xx} +(\mu''\mu^{-1}x-2\lambda
\mu'f'-\lambda\mu f'')(\underbrace{V_0}+\phi_1+\phi)_x
 \\
    &   &   \\
    &   &\left.\underbrace{+2\mu''\mu^{-1}-2(\mu')^2\mu^{-2}}\right]-k\lambda\mu(\underbrace{\phi_1}+\phi)_x-k^2(x+\lambda\mu
    f)(\underbrace{V_0}+\phi_1+\phi)_x\\
    &&\\
    &&
-k\lambda\mu\left[\frac{1}{1-(\lambda^{-1}\mu^{-1}x+f)k}-1-(\lambda^{-1}\mu^{-1}x+f)k\right]({V_0}+\phi_1+\phi)_x
  \\
    &   &   \\
    &
    &+\frac{(\lambda^{-1}\mu^{-1}x+f)k'}{[1-(\lambda^{-1}\mu^{-1}x+f)k]^3}\left[(\phi_1+\phi)_s+(\mu'\mu^{-1}x-\lambda
\mu
f')(V_0+\phi_1+\phi)_x+2\mu'\mu^{-1}\right]+\underbrace{\lambda^2\mu^2e^{V_0}\frac{\phi_1^2}{2}}\\
& & \\
& &  +\lambda^2\mu^2\left(e^{V_0+\phi_1}-e^{V_0}-e^{V_0}\phi_1
-e^{V_0}\frac{\phi_1^2}{2}\right)+\lambda^2\mu^2(e^{V_0+\phi_1+\phi}-e^{V_0+\phi_1}-e^{V_0+\phi_1}\phi)\\
& &\\
& & +\lambda^2\mu^2(e^{V_0+\phi_1}-e^{V_0})\phi,
\end{array}
\end{equation}
where we have underbraced higher order terms.

Motivated from the above relation, and recalling the equation
satisfied by $V_0$, we choose $\phi=\phi_2(s,x)$ to satisfy, for
fixed $s\in [0,\ell]$,
\begin{equation}\label{eqphi(2)}
\begin{array}{lll}
  \phi_{xx}+e^{V_0}\phi & = & \lambda^{-2}\mu^{-2}(\mu'\mu^{-1}x-\lambda\mu
f')^2e^{V_0}-\lambda^{-2}\mu^{-2}(\mu''\mu^{-1}x-2\lambda\mu'f'-\lambda\mu
f'')(V_0)_x \\
    &   &   \\
    &   & -2\lambda^{-2}\mu''\mu^{-3}+2\lambda^{-2}(\mu')^2\mu^{-4}+\lambda^{-1}\mu^{-1}k(\phi_1)_x\\
    &   &   \\
    &   &+\lambda^{-2}\mu^{-2}k^2(x+\lambda\mu
f)(V_0)_x-\frac{1}{2}e^{V_0}\phi_1^2,
\end{array}
\end{equation}
$x\in \mathbb{R}$.

Arguing as in the proof of Lemma \ref{lemphi1} we have:
\begin{lemma}\label{lemphi2}
Equation (\ref{eqphi(2)}) has a family of solutions such that
\[\begin{array}{llll}
\phi_2(s,x) &=  &
\left(-\frac{\sqrt{2}}{6}\lambda^{-2}\mu''\mu^{-3}+\frac{\sqrt{2}}{3}\lambda^{-2}\mu^{-2}k^2\right)x^3+\left(-\lambda^{-2}\mu''
\mu^{-3}+\lambda^{-2}(\mu')^2\mu^{-4}+k^2\lambda^{-2}\mu^{-2}\right.\\
& & \\
&
&\left.-\frac{1}{\sqrt{2}}\lambda^{-2}\mu^{-2}\Delta_1k+\sqrt{2}\lambda^{-1}\mu^{-2}\mu'f'+\frac{1}{\sqrt{2}}\lambda^{-1}\mu^{-1}f''+\frac{1}{\sqrt{2}}\lambda^{-1}\mu^{-1}k^2f\right)x^2\\
& & \\
& & +\left(-\sqrt{2}\Delta_2+B_2\right)
x+A_2-2\Delta_2-E_2+\sum_{i=1}^{2}(|\Delta_i|+|E_i|+1)^2\mathcal{O}(e^{-c|x|}),
 \end{array}
\]
as $x\to -\infty$, and
\[\begin{array}{llll}
\phi_2(s,x) &=  &
\left(\frac{\sqrt{2}}{6}\lambda^{-2}\mu''\mu^{-3}-\frac{\sqrt{2}}{3}\lambda^{-2}\mu^{-2}k^2\right)x^3+\left(-\lambda^{-2}\mu''
\mu^{-3}+\lambda^{-2}(\mu')^2\mu^{-4}+k^2\lambda^{-2}\mu^{-2}\right.\\
& & \\
&
&\left.+\frac{1}{\sqrt{2}}\lambda^{-2}\mu^{-2}\Delta_1k-\sqrt{2}\lambda^{-1}\mu^{-2}\mu'f'-\frac{1}{\sqrt{2}}\lambda^{-1}\mu^{-1}f''-\frac{1}{\sqrt{2}}\lambda^{-1}\mu^{-1}k^2f\right)x^2\\
& & \\
& & +\left(\sqrt{2}\Delta_2
+B_2\right)x+A_2-2\Delta_2+E_2+\sum_{i=1}^{2}(|\Delta_i|+|E_i|+1)^2\mathcal{O}(e^{-c|x|}),
 \end{array}
\]
as $x\to \infty$, where
\[\begin{array}{l}
A_2=a_1(f')^2+a_2\lambda^{-2}\mu''\mu^{-3}+a_3\lambda^{-2}\mu^{-2}k^2+a_4\lambda^{-2}\mu^{-2}\Delta_1^2
+a_5\lambda^{-2}\mu^{-2}E_1^2+a_6\lambda^{-2}\mu^{-2}kE_1,
 \\
      \\
    B_2=b_1\lambda^{-1}\mu^{-1}f''+b_2\lambda^{-1}\mu^{-2}\mu'f'+b_3\lambda^{-2}\mu^{-2}k\Delta_1+b_4\lambda^{-1}\mu^{-1}k^2f+b_5
\lambda^{-2}\mu^{-2}\Delta_1E_1,
  \end{array}
\]
for some known constants $a_i,b_i$ (independent of $\lambda$),
determined by integrals of known functions,
 for every $\Delta_i,\ E_i \in
C_{per}^2([0,\ell]),\ i=1,2$ ($\Delta_1,E_1$ determine $\phi_1$ by
Lemma \ref{lemphi2}, note that here we opted not to place the
coefficient  $\lambda^{-2}\mu^{-2}$ in front of $\Delta_2, E_2$).
Moreover, the above estimates can be differentiated.\end{lemma}

In view of the above lemma, and (\ref{equ2}), we have
\begin{equation}\label{equ2-}
\begin{array}{lll}
  u_2(s,-\lambda^{-1}\mu^{-1}L) & = & u_1(s,-\lambda^{-1}\mu^{-1}L)-\frac{\sqrt{2}}{6}(2\lambda^{-2}\mu^{-2}k^2-\lambda^{-2}\mu^{-3}\mu'')(L+\lambda\mu
  f)^3 \\
 && \\
    &&+\left(-\lambda^{-2}\mu''
\mu^{-3}+\lambda^{-2}(\mu')^2\mu^{-4}+k^2\lambda^{-2}\mu^{-2}
-\frac{1}{\sqrt{2}}\lambda^{-2}\mu^{-2}\Delta_1k+\sqrt{2}\lambda^{-1}\mu^{-2}\mu'f'\right.\\
&& \\
&&
\left.+\frac{1}{\sqrt{2}}\lambda^{-1}\mu^{-1}f''+\frac{1}{\sqrt{2}}\lambda^{-1}\mu^{-1}k^2f\right)(L+\lambda\mu
f)^2\\
&& \\
&&+(\sqrt{2}\Delta_2-B_2)(L+\lambda\mu f)+A_2-2\Delta_2-E_2\\
&&\\
    &&+\sum_{i=1}^{2}(|\Delta_i|+|E_i|+1)^2e^{-c(L+\lambda\mu
    f)},
\end{array}
\end{equation}
and
\begin{equation}\label{equ2+}
\begin{array}{lll}
   u_2(s,\lambda^{-1}\mu^{-1}L) & = & u_1(s,\lambda^{-1}\mu^{-1}L)
   +\frac{\sqrt{2}}{6}(-2\lambda^{-2}\mu^{-2}k^2+\lambda^{-2}\mu^{-3}\mu'')(L-\lambda\mu
  f)^3 \\
 && \\
    &&+\left(-\lambda^{-2}\mu''
\mu^{-3}+\lambda^{-2}(\mu')^2\mu^{-4}+k^2\lambda^{-2}\mu^{-2}
+\frac{1}{\sqrt{2}}\lambda^{-2}\mu^{-2}\Delta_1k-\sqrt{2}\lambda^{-1}\mu^{-2}\mu'f'\right.\\
&& \\
&&
\left.-\frac{1}{\sqrt{2}}\lambda^{-1}\mu^{-1}f''-\frac{1}{\sqrt{2}}\lambda^{-1}\mu^{-1}k^2f\right)(L-\lambda\mu
f)^2\\
&& \\
&&+(\sqrt{2}\Delta_2+B_2)(L-\lambda\mu f)+A_2-2\Delta_2+E_2\\
&&\\
    &&+\sum_{i=1}^{2}(|\Delta_i|+|E_i|+1)^2e^{-c(L-\lambda\mu
    f)},
\end{array}
\end{equation}
\section{The outer approximations} In this section we will construct
outer approximations $W^\pm$, valid in $\Omega^\pm$ respectively,
with the following properties: $W^\pm$ are harmonic in
$\Omega^\pm$ with zero boundary conditions on $\partial \Omega^\pm
\cap
\partial\Omega$, and they match, in the $C^1$-sense with the  inner
approximation $u_i$ of $i+1$  order,  close to the curve $\gamma$.
\subsection{The second order outer approximation}
By virtue of (\ref{eqmu,f-norms}), (\ref{eqL}), we can expand
\begin{equation}\label{eqW+-}
\begin{array}{lll}
  W^+(s,-\lambda^{-1}\mu^{-1} L) & = & W^+(s,0)-\left(\partial_nW^+(s,0) \right)\lambda^{-1}
\mu^{-1}L+\frac{1}{2}W_{tt}^+(s,0)\lambda^{-2}\mu^{-2}L^2 \\
    &   &   \\
    &   & -\frac{1}{6}W_{ttt}^+(s,0)\lambda^{-3}\mu^{-3}L^3+\mathcal{O}(W_{ttt}^+)\lambda^{-4}\mu^{-4}L^4,
\end{array}
\end{equation}
and
\begin{equation}\label{eqW-+}
\begin{array}{lll}
  W^-(s,+\lambda^{-1}\mu^{-1} L) & = & W^-(s,0)+\left(\partial_nW^-(s,0) \right)\lambda^{-1}
\mu^{-1}L+\frac{1}{2}W_{tt}^-(s,0)\lambda^{-2}\mu^{-2}L^2 \\
    &   &   \\
    &   &+\frac{1}{6}W_{ttt}^-(s,0)\lambda^{-3}\mu^{-3}L^3+\mathcal{O}(W_{ttt}^-)\lambda^{-4}\mu^{-4}L^4.
\end{array}
\end{equation}
From (\ref{equ1-}) and (\ref{eqW+-}), we find that
\begin{equation}\label{equ1-W+-}
\begin{array}{lll}
  (u_1-W^+)(s,-\lambda^{-1}\mu^{-1}L) & = & 2\ln(2\mu)-\sqrt{2}\lambda\mu f-W^+(s,0) \\
    &   &   \\
    &   & +\lambda^{-1}\mu^{-1}\left(\frac{k}{\sqrt{2}}\lambda^2\mu^2f^2-2k\lambda\mu f+\sqrt{2}\Delta_1 \lambda\mu f-2\Delta_1-E_1 \right)  \\
    &   &   \\
    &   & +\lambda^{-1}\mu^{-1}L\left(-\sqrt{2}\lambda\mu+\sqrt{2}k\lambda\mu f-2k+\sqrt{2}\Delta_1+\partial _nW^+(s,0) \right)  \\
    &   &   \\
    &   &  + \lambda^{-1}\mu^{-1}L^2\left(\frac{k}{\sqrt{2}}-\frac{1}{2}W^+_{tt}(s,0)\lambda^{-1}\mu^{-1} \right)\\
   &   &   \\
    &   &
    -\mathcal{O}(W_{ttt}^+)\lambda^{-3}\mu^{-3}L^3+\lambda^{-1}\mu^{-1}\left(1+|\Delta_1|+|E_1|
    \right)e^{-c(L+\lambda\mu f)}.
\end{array}
\end{equation}
Similarly, via (\ref{equ1+}) and (\ref{eqW-+}), we have
\begin{equation}\label{equ1-W-+}
\begin{array}{lll}
  (u_1-W^-)(s,\lambda^{-1}\mu^{-1}L) & = & 2\ln(2\mu)+\sqrt{2}\lambda\mu f-W^-(s,0) \\
    &   &   \\
    &   & +\lambda^{-1}\mu^{-1}\left(-\frac{k}{\sqrt{2}}\lambda^2\mu^2f^2-2k\lambda\mu f-\sqrt{2}\Delta_1 \lambda\mu f-2\Delta_1+E_1 \right)  \\
    &   &   \\
    &   & +\lambda^{-1}\mu^{-1}L\left(-\sqrt{2}\lambda\mu+\sqrt{2}k\lambda\mu f+2k+\sqrt{2}\Delta_1-\partial _nW^-(s,0) \right)  \\
    &   &   \\
    &   &  + \lambda^{-1}\mu^{-1}L^2\left(-\frac{k}{\sqrt{2}}-\frac{1}{2}W^-_{tt}(s,0)\lambda^{-1}\mu^{-1} \right)\\
   &   &   \\
    &   &
    -\mathcal{O}(W_{ttt}^-)\lambda^{-3}\mu^{-3}L^3+\lambda^{-1}\mu^{-1}\left(1+|\Delta_1|+|E_1|
    \right)e^{-c(L-\lambda\mu f)}.
\end{array}
\end{equation}

We claim that there exist harmonic functions $W^\pm$ in
$\Omega^\pm$, satisfying Dirichlet boundary conditions on
$\partial\Omega^\pm\cap\partial\Omega$, and smooth functions
$f,\mu>0,\Delta_1,E_1$ on $\gamma$, satisfying
(\ref{eqmu,f-norms}), such that the following hold: The first two
lines of the righthand side of (\ref{equ1-W+-}), (\ref{equ1-W-+})
respectively vanish, and the third and fourth vanish up to main
order.

%Firstly, we choose $\Delta_1, E_1$ such that
%\begin{equation}\label{eqDelta1--E1}
%\Delta_1=-k\lambda\mu f\ \ \textrm{and}\ \
%E_1=\frac{k}{\sqrt{2}}\lambda^2\mu^2f^2+\sqrt{2}\Delta_1 \lambda\mu
%f=-\frac{1}{\sqrt{2}}k\lambda^2\mu^2f^2.
%\end{equation}
%This makes the second line of (\ref{equ1-W+-}), (\ref{equ1-W-+})
%respectively to be zero. Note that $\mu,\ f,\ W^\pm$ remain to be
%chosen.
%We then choose
%\begin{equation}\label{eqmu1}
%\mu_1=\pm \frac{\sqrt{2}}{\lambda}\left(\ln
%\frac{\sqrt{2}}{\lambda}\right)\partial_nW_\gamma^\pm\ \
%\textrm{on}\ \ \gamma.
%\end{equation}
%The above choice is motivated by formally matching the first order
%inner approximation (\ref{equ0}) with the first order outer ones
%$\approx\Gamma W_\gamma^\pm$, for an appropriate $\Gamma \in
%\mathbb{R}$.
We seek suitable $W^\pm$ in the form
\[
W_1^\pm=\Gamma_1W_\gamma^\pm+{w}_0^\pm+{w}_1^\pm+{w}_2^\pm,
\]
with $\Gamma_1\in \mathbb{R}$, and ${w}_i^\pm$ satisfying
\begin{equation}\label{eqwi}
\left\{\begin{array}{ll}
         \Delta w_i^\pm=0 & \textrm{in}\ \ \Omega^\pm, \\
           &   \\
         w_i^\pm=0  & \textrm{on}\ \ \partial\Omega^\pm \cap
         \partial\Omega,
       \end{array}
 \right.
\end{equation}
$i=0,1,2$. In view of the third lines of (\ref{equ1-W+-}),
(\ref{equ1-W-+}), we choose $w_0^\pm$ by imposing that
\begin{equation}\label{eqw0-boundary}
\partial_n w^\pm_0=2k\ \ \textrm{on}\ \ \gamma.
\end{equation}
By Proposition \ref{proW}, we infer that
\begin{equation}\label{eqw0+=-w0-}
w_0^+=-w_0^-\ \ \textrm{on}\ \ \gamma.
\end{equation}
Therefore we can choose $f$ as
\begin{equation}\label{eqf1}
f_1=\mp \frac{1}{\sqrt{2}}\frac{w_0^\pm}{\lambda\mu}\ \
\textrm{on}\ \ \gamma.
\end{equation}
Writing equations $\Delta W_\gamma^\pm=0$ in coordinates $(s,t)$,
as in (\ref{eqS(u)}), and recalling that $W_\gamma^\pm(s,0)=1,\
s\in [0,\ell]$, we see that
\begin{equation}\label{eqWgamma-on}
(W_\gamma^\pm)_{tt}-k(W_\gamma^\pm)_t=0\ \ \textrm{on}\ \ \gamma.
\end{equation}
So,
\[
\pm\frac{k}{\sqrt{2}}-\frac{1}{2}W^\pm_{tt}\lambda^{-1}\mu^{-1}=
\pm\frac{k}{\sqrt{2}}-\frac{\Gamma_1}{2}k(\partial_nW_\gamma^\pm)\lambda^{-1}\mu^{-1}-\frac{1}{2}\left[(w_0^\pm)_{tt}
+(w_1^\pm)_{tt}+(w_2^\pm)_{tt}\right]\lambda^{-1}\mu^{-1}\ \
\textrm{on}\ \gamma.
\]
In view of the third and fourth line of the righthand side of
(\ref{equ1-W+-}), (\ref{equ1-W-+}) respectively, and the above
relation, given $\Gamma_1$, we can choose $\mu=\mu_1$ such that
\begin{equation}\label{eqmu1}
\lambda\mu_1=\pm
\frac{\Gamma_1}{\sqrt{2}}(\partial_nW_\gamma^\pm)\ \ \textrm{on}\
\ \gamma,
\end{equation}
(keep in mind (\ref{eqHopf})). This is indeed possible by property
(\ref{eqHarmonic-measure-derivs}). In view of the resulting first
line of (\ref{equ1-W+-}), (\ref{equ1-W-+}), we choose $\Gamma_1$
such that
\begin{equation}\label{eqGamma1}
2\ln\frac{\sqrt{2}}{\lambda}+2\ln\Gamma_1=\Gamma_1,
\end{equation}
and $w_1^\pm$ by imposing that
\begin{equation}\label{eqw1}
w_1^\pm=2\ln(\pm \partial_n W_\gamma^\pm)\ \ \textrm{on}\ \
\gamma.
\end{equation}

%\begin{equation}\label{eqmu1,gamma1}
%w_1^\pm
%=2\ln(2\mu_1)=2\ln\frac{\sqrt{2}}{\lambda}+2\ln\left(\ln\frac{\sqrt{2}}{\lambda}
%\right)+2\ln (\pm 2 \partial_nW^\pm_\gamma)\ \ \textrm{on}\ \
%\gamma.
%\end{equation}
%With this choice, the first lines of the righthand sides of
%(\ref{equ1-W+-}) and (\ref{equ1-W-+}) have become zero. For future
%purposes, note that
%\begin{equation}\label{eqw0pm}
%w_0^\pm=\left[2\ln
%\frac{\sqrt{2}}{\lambda}+2\ln\left(\ln\frac{\sqrt{2}}{\lambda}
%\right) \right]W_\gamma^\pm+\tilde{w}_0^\pm\ \ \textrm{in}\ \
%\Omega^\pm,
%\end{equation}
%where
%\begin{equation}\label{eqw0tilda}
%\left\{
%\begin{array}{ll}
%    &   \\
% \tilde{w}_0^\pm=2\ln (\pm 2 \partial_nW^\pm_\gamma)   & \textrm{on}\ \ \gamma,  \\
%    &   \\
%   \tilde{w}_0^\pm=0 & \textrm{on}\ \ \partial \Omega^\pm \cap
%   \partial \Omega.
%\end{array}
%\right.
%\end{equation}
Clearly, if $\lambda>0$ is sufficiently small, equation
(\ref{eqGamma1}) has a unique solution
\begin{equation}\label{eqgamma1+}\Gamma_1=2\ln\frac{\sqrt{2}}{\lambda}+2\ln\left(\ln\frac{\sqrt{2}}{\lambda} \right)+2\ln2 +o(1)\ \ \textrm{as} \ \ \lambda\to
0. \end{equation} Note also that by Proposition \ref{proW}, and
(\ref{eqHarmonic-measure-derivs}), we deduce that
\begin{equation}\label{eqw1tilda-derivatives=0}
\partial_n {w}_1^++\partial_n {w}_1^-=0\ \ \textrm{on}\
\ \gamma.
\end{equation}

It remains to choose $w_2^\pm, \Delta_1$, and $E_1$. In view of
the second and third line of (\ref{equ1-W+-}), (\ref{equ1-W-+})
respectively, and relation (\ref{eqw1tilda-derivatives=0}), we
choose
\begin{equation}\label{eqDelta1--E1}
\begin{array}{l}
  \Delta_1=-k\lambda\mu_1 f_1\mp\frac{1}{\sqrt{2}}\partial_nw_1^\pm, \\
   \\
  E_1=\frac{k}{\sqrt{2}}\lambda^2\mu_1^2f_1^2+\sqrt{2}\Delta_1
  \lambda\mu_1
f_1=-\frac{1}{\sqrt{2}}k\lambda^2\mu_1^2f_1^2 \mp  \lambda\mu_1
f_1\partial_nw_1^\pm.
\end{array}
\end{equation}
With the above choices, relations  (\ref{equ1-W+-}),
(\ref{equ1-W-+}) have simplified to
\begin{equation}\label{equ1-W+-new}
\begin{array}{lll}
  (u_1-W_1^+)(s,-\lambda^{-1}\mu_1^{-1}L) & = & -w_2^+(s,0)+\sqrt{2}\lambda^{-1}\mu_1^{-1}\partial_nw_1^
  +(s,0)+\lambda^{-1}\mu_1^{-1}L\partial_nw_2^+(s,0) \\
    &   &   \\
&   &  - \frac{1}{2}\lambda^{-2}\mu_1^{-2}L^2\left[(w_0^+)_{tt}
+(w_1^+)_{tt}+(w_2^+)_{tt}\right](s,0)\\
   &   &   \\
    &   &
    -\mathcal{O}(W_{1,ttt}^+)\lambda^{-3}\mu_1^{-3}L^3+\lambda^{-1}\mu_1^{-1}\left(1+|\Delta_1|+|E_1|
    \right)e^{-c(L+\lambda\mu_1 f)},
\end{array}
\end{equation}
and
\begin{equation}\label{equ1-W-+new}
\begin{array}{lll}
  (u_1-W_1^-)(s,\lambda^{-1}\mu^{-1}L) & = & -w_2^-(s,0)-\sqrt{2}\lambda^{-1}\mu_1^{-1}\partial_nw_1^-(s,0)
  -\lambda^{-1}\mu_1^{-1}L\partial_nw_2^-(s,0)\\
       &   &   \\
    &   &   -\frac{1}{2}\lambda^{-2}\mu_1^{-2}L^2\left[(w_0^-)_{tt}
+(w_1^-)_{tt}+(w_2^-)_{tt}\right](s,0)\\
   &   &   \\
    &   &
    -\mathcal{O}(W_{1,ttt}^-)\lambda^{-3}\mu_1^{-3}L^3+\lambda^{-1}\mu_1^{-1}\left(1+|\Delta_1|+|E_1|
    \right)e^{-c(L-\lambda\mu_1 f)},
\end{array}
\end{equation}
respectively. Finally, we choose $w_2^\pm$ by imposing that
\begin{equation}\label{eqw2}
w_2^\pm=\pm\sqrt{2}\lambda^{-1}\mu_1^{-1}\partial_nw_1^
  \pm=\frac{2}{\Gamma_1}\frac{\partial_nw_1^\pm}{\partial_nW_\gamma^\pm}\ \ \textrm{on}\ \ \gamma.
\end{equation}
By Theorem 6.6 in \cite{gilbarg} or Theorem 2I in \cite{miranda},
(\ref{eqw1}), (\ref{eqgamma1+}), and the smoothness of the curve
$\gamma$, we infer that
\begin{equation}\label{eqw2holder}
\|w_2^\pm\|_{C^{m}(\bar{\Omega}^\pm)}\leq
C_m\left(\ln\frac{1}{\lambda} \right)^{-1},\ \ m\geq 0.
\end{equation}
For future reference note that, thanks to Proposition \ref{proW},
(\ref{eqw1tilda-derivatives=0}), and  (\ref{eqw2}), we have
\begin{equation}\label{eqw2-derivatives=0}
\partial_n {w}_2^++\partial_n {w}_2^-=0\ \ \textrm{on}\
\ \gamma.
\end{equation}

 By (\ref{eqL}), (\ref{eqf1}), (\ref{eqgamma1+}),
(\ref{equ1-W+-new}), (\ref{equ1-W-+new}), and (\ref{eqw2}),
(keeping in mind (\ref{eqw2holder})), we infer that
\begin{equation}\label{equ1-W1}
\begin{array}{lll}
  \left|(u_1-W^\pm_1)(s,\mp\lambda^{-1}\mu^{-1}L)\right|\leq C

 \left[\ln\left(\ln\frac{1}{\lambda}\right)\right]^3
 \left(\ln\frac{1}{\lambda}\right)^{-2},
\end{array}
\end{equation}
for $\lambda>0$ sufficiently small (having increased the value of
$M$ so that the exponential terms become negligible).
\begin{remark}\label{remNonlinesar-boundary}
In view of the first and third lines of (\ref{equ1-W+-}),
(\ref{equ1-W-+}), one is at first tempted to choose $\mu$,
$w^\pm=W^\pm-w_0^\pm$ such that
\[
2\ln(2\mu)=w^\pm\ \ \textrm{and}\ \
\sqrt{2}\lambda\mu=\pm\partial_nw^\pm\ \ \textrm{on}\ \ \gamma.
\]
This leads to the following semilinear boundary value problems:
\[
\left\{
\begin{array}{ll}
  \Delta w^\pm=0 & \textrm{in}\ \ \Omega^\pm, \\
    &   \\
  w^\pm=0 & \textrm{on}\ \ \partial\Omega^\pm\cap\partial\Omega,\\
& \\
\partial_nw^\pm=\pm\frac{\lambda}{\sqrt{2}}e^\frac{w^\pm}{2}& \textrm{on}\ \
\gamma.
\end{array}
\right.
\]
We are looking for ``large" (maximal) solutions $w^\pm$, so that
(\ref{eqmu,f-norms}) holds. This is a problem which seems to have
its own independent interest, and we expect that it shares common
features with the original problem (\ref{eqEqradialmain}).
\end{remark}
\subsection{The third order outer approximation}
In this subsection, we will choose functions $\mu, f$ satisfying
(\ref{eqmu,f-norms}), functions $\Delta_i, E_i$, $i=1,2$
($\Delta_1, E_1$ possibly new, though they will turn out to be
chosen the same as before), and harmonic functions $W^{\pm}$, so
that the inner approximation $u_2$ matches, in the $C^1$ sense,
with the outer ones $W^{\pm}$ at $\mp \lambda^{-1}\mu^{-1}L$
respectively, with $L$ satisfying (\ref{eqL}), as $\lambda\to 0$.

 From (\ref{equ1-}),
(\ref{equ2-}), and (\ref{eqW+-}), we find that
\begin{equation}\label{equ2-W+-}
\begin{array}{c}
  (u_2-W^+)(s,-\lambda^{-1}\mu^{-1}L)= \\
    \\
  \frac{\sqrt{2}}{6}\left(\lambda^{-2}\mu^{-3}\mu''-2k^2\lambda^{-2}\mu^{-2}
 +\frac{1}{\sqrt{2}}\lambda^{-3}\mu^{-3}W^+_{ttt}(s,0)\right)L^3 \\
    \\
  +\left(\frac{1}{\sqrt{2}}\lambda^{-1}\mu^{-1}k-\frac{1}{\sqrt{2}}\lambda^{-1}\mu^{-1}k^2f+\frac{1}{\sqrt{2}}\lambda^{-1}\mu^{-2}\mu''f
-\mu''\lambda^{-2}\mu^{-3}+(\mu')^2\lambda^{-2}\mu^{-4}\right. \\
    \\
  \left.+\lambda^{-2}\mu^{-2}k^2-\frac{1}{\sqrt{2}}\lambda^{-2}\mu^{-2}k\Delta_1
+\sqrt{2}\lambda^{-1}\mu^{-2}\mu'f'+
\frac{1}{\sqrt{2}}\lambda^{-1}\mu^{-1}f''-\frac{1}{2}\lambda^{-2}\mu^{-2}W^+_{tt}(s,0)\right)L^2 \\
    \\
  +\left(-\sqrt{2}+\sqrt{2}kf-2\lambda^{-1}\mu^{-1}k+\sqrt{2}\lambda^{-1}\mu^{-1}\Delta_1+\frac{\sqrt{2}}{2}\mu^{-1}\mu''f^2
-2\lambda^{-1}\mu^{-2}\mu''f+2(\mu')^2\lambda^{-1}\mu^{-3}f \right. \\
    \\
  +2\lambda^{-1}\mu^{-1}k^2f-\sqrt{2}\lambda^{-1}\mu^{-1}\Delta_1kf+2\sqrt{2}\mu^{-1}\mu'f'f+\sqrt{2}f''f
+\sqrt{2}\Delta_2-b_1\lambda^{-1}\mu^{-1}f'' \\
    \\
  \left.
-b_2\lambda^{-1}\mu^{-2}\mu'f'-b_3\lambda^{-2}\mu^{-2}\Delta_1k-b_4\lambda^{-1}\mu^{-1}k^2f
-b_5\lambda^{-2}\mu^{-2}\Delta_1E_1+\lambda^{-1}\mu^{-1}\partial_nW^+(s,0)\right)L \\
   \\
 +2\ln(2\mu)-\sqrt{2}\lambda\mu f+\frac{1}{\sqrt{2}}\lambda\mu
kf^2-2kf+\sqrt{2}\Delta_1 f- 2\lambda^{-1}\mu^{-1}\Delta_1-
\lambda^{-1}\mu^{-1}E_1-\frac{\sqrt{2}}{3}\lambda \mu k^2
f^3\\
    \\
  +\frac{\sqrt{2}}{6}\lambda\mu'' f^3 -\mu''\mu^{-1}f^2+(\mu')^2\mu^{-2}f^2+k^2f^2-\frac{1}{\sqrt{2}}\Delta_1kf^2+\sqrt{2}\lambda\mu'f'f^2+\frac{1}{\sqrt{2}}
\lambda\mu f''f^2 \\
   \\
+\frac{1}{\sqrt{2}}\lambda\mu
k^2f^3+\sqrt{2}\Delta_2 \lambda\mu f -b_1f''f-b_2\mu^{-1}\mu'f'f-b_3\lambda^{-1}\mu^{-1}\Delta_1kf-b_4k^2f^2-b_5\lambda^{-1}\mu^{-1}\Delta_1E_1f \\
    \\
  +a_1(f')^2+a_2\lambda^{-2}\mu^{-3}\mu''+a_3\lambda^{-2}\mu^{-2}k^2+a_4\lambda^{-2}\mu^{-2}\Delta_1^2+a_5\lambda^{-2}\mu^{-2}E_1^2+a_6\lambda^{-2}\mu^{-2}kE_1-2\Delta_2
 \\
    \\
   -E_2-W^+(s,0) +\mathcal{O}(W^+_{tttt})\lambda^{-4}\mu^{-4}L^4+\sum_{i=1}^{2}(|\Delta_i|+|E_i|+1)^2e^{-c(L+\lambda\mu
    f)}.
\end{array}
\end{equation}
Similarly, via (\ref{equ1+}), (\ref{equ2+}), and (\ref{eqW-+}), we
obtain that
\begin{equation}\label{equ2-W-+}
\begin{array}{c}
  (u_2-W^-)(s,\lambda^{-1}\mu^{-1}L)= \\
    \\
  \frac{\sqrt{2}}{6}\left(\lambda^{-2}\mu^{-3}\mu''-2k^2\lambda^{-2}\mu^{-2}
 -\frac{1}{\sqrt{2}}\lambda^{-3}\mu^{-3}W^-_{ttt}(s,0)\right)L^3 \\
    \\
  +\left(-\frac{1}{\sqrt{2}}\lambda^{-1}\mu^{-1}k+\frac{1}{\sqrt{2}}\lambda^{-1}\mu^{-1}k^2f-\frac{1}{\sqrt{2}}\lambda^{-1}\mu^{-2}\mu''f
-\mu''\lambda^{-2}\mu^{-3}+(\mu')^2\lambda^{-2}\mu^{-4}\right. \\
    \\
  \left.+\lambda^{-2}\mu^{-2}k^2+\frac{1}{\sqrt{2}}\lambda^{-2}\mu^{-2}k\Delta_1
-\sqrt{2}\lambda^{-1}\mu^{-2}\mu'f'-
\frac{1}{\sqrt{2}}\lambda^{-1}\mu^{-1}f''-\frac{1}{2}\lambda^{-2}\mu^{-2}W^-_{tt}(s,0)\right)L^2 \\
    \\
  +\left(-\sqrt{2}+\sqrt{2}kf+2\lambda^{-1}\mu^{-1}k+\sqrt{2}\lambda^{-1}\mu^{-1}\Delta_1+\frac{\sqrt{2}}{2}\mu^{-1}\mu''f^2
+2\lambda^{-1}\mu^{-2}\mu''f-2(\mu')^2\lambda^{-1}\mu^{-3}f \right. \\
    \\
  -2\lambda^{-1}\mu^{-1}k^2f-\sqrt{2}\lambda^{-1}\mu^{-1}\Delta_1kf+2\sqrt{2}\mu^{-1}\mu'f'f+\sqrt{2}f''f
+\sqrt{2}\Delta_2+b_1\lambda^{-1}\mu^{-1}f'' \\
    \\
  \left.
+b_2\lambda^{-1}\mu^{-2}\mu'f'+b_3\lambda^{-2}\mu^{-2}\Delta_1k+b_4\lambda^{-1}\mu^{-1}k^2f
+b_5\lambda^{-2}\mu^{-2}\Delta_1E_1-\lambda^{-1}\mu^{-1}\partial_nW^-(s,0)\right)L \\
   \\
 +2\ln(2\mu)+\sqrt{2}\lambda\mu f-\frac{1}{\sqrt{2}}\lambda\mu
kf^2-2kf-\sqrt{2}\Delta_1 f- 2\lambda^{-1}\mu^{-1}\Delta_1+
\lambda^{-1}\mu^{-1}E_1+\frac{\sqrt{2}}{3}\lambda \mu k^2
f^3\\
    \\
  -\frac{\sqrt{2}}{6}\lambda\mu'' f^3 -\mu''\mu^{-1}f^2+(\mu')^2\mu^{-2}f^2+k^2f^2+\frac{1}{\sqrt{2}}\Delta_1kf^2
  -\sqrt{2}\lambda\mu'f'f^2-\frac{1}{\sqrt{2}}
\lambda\mu f''f^2 \\
   \\
-\frac{1}{\sqrt{2}}\lambda\mu
k^2f^3-\sqrt{2}\Delta_2 \lambda\mu f -b_1f''f-b_2\mu^{-1}\mu'f'f-b_3\lambda^{-1}\mu^{-1}\Delta_1kf-b_4k^2f^2-b_5\lambda^{-1}\mu^{-1}\Delta_1E_1f \\
    \\
  +a_1(f')^2+a_2\lambda^{-2}\mu^{-3}\mu''+a_3\lambda^{-2}\mu^{-2}k^2+a_4\lambda^{-2}\mu^{-2}\Delta_1^2+a_5\lambda^{-2}\mu^{-2}E_1^2+a_6\lambda^{-2}\mu^{-2}kE_1-2\Delta_2
 \\
    \\
  +E_2-W^-(s,0) +\mathcal{O}(W^-_{tttt})\lambda^{-4}\mu^{-4}L^4+\sum_{i=1}^{2}(|\Delta_i|+|E_i|+1)^2e^{-c(L-\lambda\mu
    f)}.
\end{array}
\end{equation}

We seek suitable $W^\pm$ in the form
\begin{equation}\label{eqWpm}
W_2^\pm=\Gamma_2W_\gamma^\pm+{w}_0^\pm+{w}_1^\pm+{w}_2^\pm+w_3^\pm+w_4^\pm,
\end{equation}
with $\Gamma_2\in \mathbb{R}$, and ${w}_i^\pm$, $i=0,\cdots,4$,
satisfying (\ref{eqwi}) but not necessarily equal to those
determined in the previous subsection (as it turns out, $w_0^\pm$,
$w_1^\pm$ and $w_2^\pm$ will be chosen the same as before). We
choose
% We (again) choose
%$w_0^{\pm}$,  $f=f_1$, $\mu=\mu_1$, $w_1^\pm$, $\Gamma=\Gamma_1$,
%$\Delta_1$, and $E_1$ as in (\ref{eqw0-boundary}), (\ref{eqf1}),
%(\ref{eqmu1}), (\ref{eqw1}), (\ref{eqGamma1}), and
%(\ref{eqDelta1--E1}) respectively. For the reader's convenience, we
%write here that
\begin{equation}\label{eqw0third}
\partial_n w_0^\pm=2k,
\end{equation}

\begin{equation}\label{eqGamma1third}
\Gamma_2=2\ln\left(\frac{\sqrt{2}}{\lambda} \right)+2\ln\Gamma_2,\
\ (\textrm{i.e.} \ \Gamma_2=\Gamma_1),
\end{equation}
recall (\ref{eqgamma1+}),
\begin{equation}\label{eqmuthird}
\lambda \mu=\pm \frac{\Gamma_2}{\sqrt{2}}\partial_nW_\gamma^\pm,
\end{equation}

\begin{equation}\label{eqw1third}
w_1^\pm=2\ln(\pm \partial_nW_\gamma^\pm).
\end{equation}
Note that, by Proposition \ref{proW}, we have
\[
w_0^++w_0^-=0\ \ \textrm{and}\ \partial_nw_1^++\partial_nw_1^-=0\
\ \textrm{on}\ \ \gamma.
\]
Given $f$, keeping in mind the above relations, we determine
$\Delta_1$, $E_1$, $E_2$, and $w_2^\pm$ from the following
relations on $\gamma$:
\begin{equation}\label{eqDelta1third}
\sqrt{2}kf+\sqrt{2}\lambda^{-1}\mu^{-1}\Delta_1\pm
\lambda^{-1}\mu^{-1}\partial_n w_1^\pm=0,
\end{equation}

\begin{equation}\label{eqE1third}
E_1=\frac{1}{\sqrt{2}}\lambda^2\mu^2kf^2+\sqrt{2}\Delta_1\lambda
\mu f,
%\mp \lambda\mu w_2^\pm
\end{equation}
\begin{equation}\label{eqE2third}\begin{array}{lll}
    E_2 & = & -\frac{\sqrt{2}}{3}\lambda\mu k^2f^3+\frac{\sqrt{2}}{6}\lambda \mu''
f^3-\frac{1}{\sqrt{2}}\Delta_1kf^2+\sqrt{2}\lambda\mu'f'f^2 \\
      &   &   \\
      &   & +\frac{1}{\sqrt{2}}\lambda\mu f'' f^2+\frac{1}{\sqrt{2}}\lambda\mu
k^2f^3+{\sqrt{2}}\Delta_2\lambda \mu f,
  \end{array}
\end{equation}
and
\begin{equation}\label{eqw2third}
w_2^\pm=\pm\sqrt{2}\lambda^{-1}\mu^{-1}\partial_nw_1^\pm=\frac{2}{\Gamma_1}\frac{\partial_nw_1^\pm}{\partial_nW_\gamma^\pm}.
\end{equation}
We remark that the above choices of $w_0^\pm$, $w_1^\pm$,
$w_2^\pm$, $\Gamma_2$, $\Delta_1$ and $E_1$ are the same as in the
previous subsection. Again, from Proposition \ref{proW}, we find
that
\[
w_2^+=w_2^-\ \ \textrm{on}\ \ \gamma.
\]
Hence, we can choose
\begin{equation}\label{eqDelta2third}
\Delta_2=-\frac{1}{2}\mu^{-1}\mu''f^2+\lambda^{-1}\mu^{-1}\Delta_1kf-2\mu^{-1}\mu'f'f-f''f\mp\frac{1}{\sqrt{2}}\lambda^{-1}\mu^{-1}\partial_n
w_2^\pm.
\end{equation}
Next, given $w_3^\pm$ such that
\begin{equation}\label{eqw3pm=0}
w_3^++w_3^-=0\ \ \textrm{on}\ \ \gamma,
\end{equation}
 we take
\begin{equation}\label{eqfthird}
f=\mp\frac{\lambda^{-1}\mu^{-1}}{\sqrt{2}}w_0^\pm\mp\frac{\lambda^{-1}\mu^{-1}}{\sqrt{2}}w_3^\pm.
\end{equation}
Ideally, we would like $w_3^\pm$ to satisfy
\begin{equation}\label{eqw3approximate}
\begin{array}{lll}
0&= & -2\lambda^{-1}\mu^{-2}\mu''f+2(\mu')^2\lambda^{-1}\mu^{-3}f
+2\lambda^{-1}\mu^{-1}k^2f-b_1\lambda^{-1}\mu^{-1}f''  \\
   && \\
 &&-b_2\lambda^{-1}\mu^{-2}\mu'f'-b_3\lambda^{-2}\mu^{-2}\Delta_1k-b_4\lambda^{-1}\mu^{-1}k^2f
-b_5\lambda^{-2}\mu^{-2}\Delta_1E_1   \\
  && \\
&&+\lambda^{-1}\mu^{-1}\partial_nw_3^\pm.
\end{array}
\end{equation}
Observe that the above problem is actually a \emph{nonlocal}
differential equation on $\gamma$ (keep in mind (\ref{eqfthird})).
Fortunately, the approximate solutions $w_3^\pm$, determined by
setting
$f=f_1=\mp\frac{\lambda^{-1}\mu^{-1}}{\sqrt{2}}w_0^\pm=-\frac{1}{\Gamma_1}\frac{w_0^\pm}{\partial_nW_\gamma^\pm}$
in the above equation, turn out to be satisfactory for our
purposes. Namely, we choose $w^\pm_3$ such that
\begin{equation}\label{eqw3third}\begin{array}{lll}
    \lambda^{-1}\mu^{-1}\partial_nw_3^\pm & = & 2\lambda^{-1}\mu^{-2}\mu''f_1-2(\mu')^2\lambda^{-1}\mu^{-3}f_1
-2\lambda^{-1}\mu^{-1}k^2f_1+b_1\lambda^{-1}\mu^{-1}f''_1
 \\
      &   &   \\
      &   & +b_2\lambda^{-1}\mu^{-2}\mu'f'_1+b_3\lambda^{-2}\mu^{-2}k\left(\mp
\frac{1}{\sqrt{2}}\partial_n  w_1^\pm-k\lambda \mu
f_1\right)+b_4\lambda^{-1}\mu^{-1}k^2f_1
 \\
      &   &   \\
      &   & +b_5\lambda^{-2}\mu^{-2}\left(\frac{1}{\sqrt{2}}\lambda\mu
f_1(\partial_nw_1^\pm)^2\pm\frac{3}{2}k\lambda^2\mu^2f_1^2\partial_nw_1^\pm+\frac{1}{\sqrt{2}}k^2\lambda^3\mu^3f_1^3
\right).
  \end{array}
\end{equation}
Note that (\ref{eqw3pm=0}) holds (recall Proposition \ref{proW}).
Then, denoting the righthand side of (\ref{eqw3approximate}) by
$R(f,w_3^\pm)$, the above choice of $w_3^\pm$ yields that
\begin{equation}\label{eqRw3}
\begin{array}{lll}
    R(f,w_3^\pm) & = &
    -2\lambda^{-1}\mu^{-2}\mu''h+2(\mu')^2\lambda^{-1}\mu^{-3}h
+2\lambda^{-1}\mu^{-1}k^2h-b_1\lambda^{-1}\mu^{-1}h''
 \\
      &   &   \\
      &   &
      -b_2\lambda^{-1}\mu^{-2}\mu'h'+b_3\lambda^{-1}\mu^{-1}k^2h-b_4\lambda^{-1}\mu^{-1}k^2h
 \\
      &   &   \\
      &   & \mp\frac{3}{2}b_5k(2f_1h+h^2)\partial_nw_1^\pm
-\frac{b_5}{\sqrt{2}} k^2\lambda \mu(h^3+3f_1h^2+3f_1^2h),
  \end{array}
\end{equation}
where
\begin{equation}\label{eqhthird}
h\equiv\mp \frac{1}{\sqrt{2}}\lambda^{-1} \mu^{-1} w_3^\pm.
\end{equation}

At this point, let us make a small de-tour. Given $g\in
C^{2,\alpha}(\gamma)$, $0<\alpha<1$, we define the \emph{Dirichlet
to Neumann mappings}
\[
T_{DN}^\pm:C^{2,\alpha}(\gamma)\to C^{1,\alpha}(\gamma),
\]
by $T_{DN}^\pm g=\partial_n w^\pm$ on $\gamma$, where $w^\pm \in
C^{2,\alpha}(\bar{\Omega}^\pm)$ satisfy (\ref{eqwi}) and $w^\pm=g$
on $\gamma$. It is easy to see that $T_{DN}^\pm$ are well defined
linear bounded operators. Furthermore, we have that
$\textrm{Kernel}\left(T_{DN}^\pm \right)=0$ and thus $T_{DN}^\pm$
are invertible with bounded inverses (by the closed graph
theorem). In other words, the \emph{Neumann to Dirichlet mappings}
\[
T_{ND}^\pm\equiv
\left(T_{DN}^\pm\right)^{-1}:C^{1,\alpha}(\gamma)\to
C^{2,\alpha}(\gamma),
\]
are well defined linear bounded operators.

Now, it is easy to see that
\[
\|\partial_n w_3^\pm\|_{C^2(\gamma)}\leq C\left(\ln
\frac{1}{\lambda} \right)^{-1}.
\]
Thus,
\[
\|w_3^\pm\|_{C^{2,\alpha}(\gamma)}=\|T^\pm_{ND}(\partial_n
w_3^\pm)\|_{C^{2,\alpha}(\gamma)}\leq C \|\partial_n
w_3^\pm\|_{C^{1,\alpha}(\gamma)}\leq C\left(\ln \frac{1}{\lambda}
\right)^{-1}.
\]
By the same argument leading to (\ref{eqw2holder}), we infer that
\begin{equation}\label{eqw3glob}
\|w_3^\pm\|_{C^m(\bar{\Omega}^\pm)}\leq   C_m \left( \ln
\frac{1}{\lambda} \right)^{-1},\ \ m\geq 0.
\end{equation}
In turn, via (\ref{eqhthird}), this implies that
\[
\|h\|_{C^m(\gamma)}\leq C_m\left(\ln \frac{1}{\lambda}
\right)^{-2},\ \ m\geq 0.
\]
Hence, from (\ref{eqRw3}), we obtain that
\begin{equation}\label{eqRestim}
\|R(f,w_3^\pm)\|_{C^m(\gamma)}\leq C_m\left(\ln \frac{1}{\lambda}
\right)^{-3},\ \ m\geq 0.
\end{equation}

 Finally, we
choose \begin{equation}\label{eqw4third}\begin{array}{lll}
    w_4^\pm & = & -\mu''\mu^{-1}f^2+(\mu')^2\mu^{-2}f^2+k^2f^2 \\
      &   &   \\
      &   & -b_1f''f-b_2\mu^{-1}\mu'f'f-b_3\lambda^{-1}\mu^{-1}\Delta_1kf-b_4k^2f^2-b_5\lambda^{-1}\mu^{-1}\Delta_1E_1f \\
      &   &   \\
       &   & +a_1(f')^2+a_2\lambda^{-2}\mu^{-3}\mu''
       +a_3\lambda^{-2}\mu^{-2}k^2+a_4\lambda^{-2}\mu^{-2}\Delta_1^2+a_5\lambda^{-2}\mu^{-2}E_1^2+a_6\lambda^{-2}\mu^{-2}kE_1-2\Delta_2.
   \end{array}
\end{equation}

Recalling (\ref{eqWgamma-on}), (\ref{eqmuthird}), we see that
\begin{equation}\label{eqHartt}
-\frac{1}{2}\lambda^{-2}\mu^{-2}\Gamma_2(W_\gamma^\pm)_{tt}=\mp\frac{1}{\sqrt{2}}\lambda^{-1}\mu^{-1}k\
\ \textrm{on}\ \gamma.
\end{equation}
Since $W^\pm_\gamma$ are harmonic in $\Omega^\pm$ respectively, in
view of (\ref{eqS(u)}), near the curve $\gamma$, we have
\[
(W_\gamma^\pm)_{tt}+\frac{1}{\left(1-tk(s)\right)^2}(W_\gamma^\pm)_{ss}+\frac{tk'(s)}{\left(1-tk(s)\right)^3}(W_\gamma^\pm)_{s}
-\frac{k(s)}{1-tk(s)}(W_\gamma^\pm)_{t}=0,
\]
for $\pm t\geq 0$ small, along $\gamma$. Differentiating the above
relation with respect to $t$, we arrive at
\[\begin{array}{c}
    (W_\gamma^\pm)_{ttt}+\frac{2k}{(1-tk)^3}(W_\gamma^\pm)_{ss}+\frac{1}{(1-tk)^2}(W_\gamma^\pm)_{sst}
+\frac{3tkk'}{(1-tk)^4}(W_\gamma^\pm)_{s}+\frac{k'}{(1-tk)^3}(W_\gamma^\pm)_{s}
 \\
      \\
    -\frac{k^2}{(1-tk)^2}(W_\gamma^\pm)_{t}-\frac{k}{1-tk}(W_\gamma^\pm)_{tt}+\frac{tk'}{(1-tk)^3}(W_\gamma^\pm)_{st}=0,
  \end{array}
\]
for $\pm t\geq 0$ small, along $\gamma$. Setting $t=0$, and making
use of (\ref{eqmuthird}), we obtain that
\begin{equation}\label{eqHarttt}
\Gamma_2 (W_\gamma^\pm)_{ttt}=\mp \sqrt{2}\lambda
\mu''\pm2\sqrt{2}k^2\lambda\mu\ \ \textrm{on}\ \ \gamma.
\end{equation}

 From the fact that $w_0^\pm$ are harmonic functions,
(\ref{eqS(u)}), (\ref{eqw0third}), and (\ref{eqfthird}), we obtain
that \begin{equation}\label{eqw0w3tt}\begin{array}{rcl}
    -\frac{1}{2}\lambda^{-2}\mu^{-2}(w_0^\pm+w_3^\pm)_{tt} & = & \mp \frac{1}{\sqrt{2}}\lambda^{-1}\mu^{-2}\mu''f\mp
\frac{1}{\sqrt{2}}\lambda^{-1}\mu^{-1}f''\mp
\sqrt{2}\lambda^{-1}\mu^{-2}\mu'f'-\lambda^{-2}\mu^{-2}k^2
 \\
      &   &   \\
      &   &  +\frac{\lambda^{-2}\mu^{-2}}{2}\left[-2k\mu^{-1}\mu''
f+2k\mu^{-2}(\mu')^2f+2k^3f-b_1kf''-b_2k\mu^{-1}\mu'f'\right. \\
      &   &   \\
      &   & \left.-b_3\lambda^{-1}\mu^{-1}\Delta_1 k^2-b_4k^3f -b_5
\lambda^{-1}\mu^{-1}\Delta_1 E_1 k-kR(f,w_3^{\pm}) \right]
  \end{array}
\end{equation}
on $\gamma$. From the fact that $w_1^\pm$ are harmonic functions,
(\ref{eqS(u)}), (\ref{eqmuthird}), (\ref{eqw1third}), and
(\ref{eqDelta1third})
 we obtain that
\begin{equation}\label{eqw1tt}
-\frac{1}{2}\lambda^{-2}\mu^{-2}(w_1^\pm)_{tt}=\mu''\mu^{-3}\lambda^{-2}-(\mu')^2\mu^{-4}\lambda^{-2}\pm
\frac{1}{\sqrt{2}}k^2\lambda^{-1}\mu^{-1}f\pm
\frac{1}{\sqrt{2}}k\lambda^{-2}\mu^{-2}\Delta_1\ \ \textrm{on}\
\gamma.
\end{equation}

Combining all the above, we find that
\begin{equation}\label{equ2-W2left}
\begin{array}{c}
  (u_2-W_2^+)(s,-\lambda^{-1}\mu^{-1}L)= \\
    \\
  \frac{1}{6}\lambda^{-3}\mu^{-3}\sum_{i=0}^{4}(w_i^+)_{ttt}L^3 \\
    \\
  -\frac{1}{2}\lambda^{-2}\mu^{-2}(w_2^++w_4^+)_{tt}L^2 \\
    \\
+\left[-2k\mu^{-1}\mu''
f+2k\mu^{-2}(\mu')^2f+2k^3f-b_1kf''-b_2k\mu^{-1}\mu'f'\right. \\
       \left.-b_3\lambda^{-1}\mu^{-1}\Delta_1 k^2-b_4k^3f -b_5
\lambda^{-1}\mu^{-1}\Delta_1 E_1 k-kR(f,w_3^{+})
\right]\frac{\lambda^{-2}\mu^{-2}}{2}L^2
    \\
    \\
  +\left(R(f,w_3^+)+\lambda^{-1}\mu^{-1}\partial_nw_4^+\right)L \\
   \\
+\mathcal{O}(W^+_{tttt})\lambda^{-4}\mu^{-4}L^4+\sum_{i=1}^{2}(|\Delta_i|+|E_i|+1)^2e^{-c(L+\lambda\mu
    f)},
\end{array}
\end{equation}
and
 \begin{equation}\label{equ2-W2right}
\begin{array}{c}
  (u_2-W_2^-)(s,\lambda^{-1}\mu^{-1}L)= \\
    \\
  -\frac{1}{6}\lambda^{-3}\mu^{-3}\sum_{i=0}^{4}(w_i^-)_{ttt}L^3 \\
    \\
  -\frac{1}{2}\lambda^{-2}\mu^{-2}(w_2^-+w_4^-)_{tt}L^2 \\
    \\
+\left[-2k\mu^{-1}\mu''
f+2k\mu^{-2}(\mu')^2f+2k^3f-b_1kf''-b_2k\mu^{-1}\mu'f'\right. \\
       \left.-b_3\lambda^{-1}\mu^{-1}\Delta_1 k^2-b_4k^3f -b_5
\lambda^{-1}\mu^{-1}\Delta_1 E_1 k-kR(f,w_3^{-})
\right]\frac{\lambda^{-2}\mu^{-2}}{2}L^2
    \\
    \\
  -\left(R(f,w_3^-)+\lambda^{-1}\mu^{-1}\partial_nw_4^-\right)L \\
   \\
 +\mathcal{O}(W^-_{tttt})\lambda^{-4}\mu^{-4}L^4+\sum_{i=1}^{2}(|\Delta_i|+|E_i|+1)^2e^{-c(L-\lambda\mu
    f)}
\end{array}
\end{equation}
along $\gamma$.
\begin{remark}
In the rest of the paper, the functions $\mu$ and $f$ will be
given by (\ref{eqmuthird}) and (\ref{eqfthird}) respectively.
Furthermore, in order to avoid confusion, we remind to the reader
that the functions $\Delta_1$ and $E_1$ will be as in \emph{this}
subsection. The same remark also applies to the harmonic functions
$w_i^\pm$.
\end{remark}
\section{The remainder of the third order inner approximate solution}
In this subsection, we will estimate the remainder $S(u_2)$ that
is left in the equation by the third order inner approximate
solution $u_2$, defined in (\ref{equ2}) with $\phi=\phi_2$ as in
Lemma \ref{lemphi2}, around the curve $\gamma$. In view of
(\ref{eqS(u2)}) and (\ref{eqphi(2)}), the aforementioned remainder
reduces to
\begin{equation}\label{eqS(u2)estim}
\begin{array}{lll}
  S(u_2) & = & \frac{1}{\left[
1-(\lambda^{-1}\mu^{-1}x+f)k\right]^2}\left[(\phi_1+\phi)_{ss}+2(\mu'\mu^{-1}x-\lambda\mu
f')(\phi_1+\phi)_{sx} \right.\\
&&\\
& &\left. +(\mu'\mu^{-1}x-\lambda\mu f')^2(\phi_1+\phi)_{xx}
+(\mu''\mu^{-1}x-2\lambda \mu'f'-\lambda\mu
f'')(\phi_1+\phi)_x\right]
 \\
    &   &   \\
& &+\frac{1-\left[ 1-(\lambda^{-1}\mu^{-1}x+f)k\right]^2}{\left[
1-(\lambda^{-1}\mu^{-1}x+f)k\right]^2}\left[-(\mu'\mu^{-1}x-\lambda\mu
f')^2e^{V_0}+(\mu''\mu^{-1}x-2\lambda \mu'f'-\lambda\mu
f'')(V_0)_x\right.\\
    & &\\
    &   &\left.+2\mu''\mu^{-1}-2(\mu')^2\mu^{-2} \right]-k\lambda\mu \phi_x-k^2(x+\lambda\mu
    f)(\phi_1+\phi)_x\\
    &&\\
    &&
-k\lambda\mu\left[\frac{1}{1-(\lambda^{-1}\mu^{-1}x+f)k}-1-(\lambda^{-1}\mu^{-1}x+f)k\right]({V_0}+\phi_1+\phi)_x
  \\
    &   &   \\
    &
    &+\frac{(\lambda^{-1}\mu^{-1}x+f)k'}{[1-(\lambda^{-1}\mu^{-1}x+f)k]^3}\left[(\phi_1+\phi)_s+(\mu'\mu^{-1}x-\lambda
\mu
f')(V_0+\phi_1+\phi)_x+2\mu'\mu^{-1}\right]\\
& & \\
& &  +\lambda^2\mu^2\left(e^{V_0+\phi_1}-e^{V_0}-e^{V_0}\phi_1
-e^{V_0}\frac{\phi_1^2}{2}\right)+\lambda^2\mu^2(e^{V_0+\phi_1+\phi}-e^{V_0+\phi_1}-e^{V_0+\phi_1}\phi)\\
& &\\
& & +\lambda^2\mu^2(e^{V_0+\phi_1}-e^{V_0})\phi,
\end{array}
\end{equation}

It follows from (\ref{eqgamma1+}) and (\ref{eqmuthird}), keeping
in mind that $\Gamma_2=\Gamma_1$, that there exist constants
$C_k>0$ such that
\begin{equation}\label{eqmuUpper}
\left|\partial_s^k \mu(s) \right|\leq C_k \frac{1}{\lambda}\ln
\frac{1}{\lambda},\ \ s\in [0,\ell),\ k\geq 0,
\end{equation}
for $\lambda>0$ sufficiently small (recall also the argument
leading to (\ref{eqw2holder})). On the other side, recalling
(\ref{eqHopf}), we find that
\begin{equation}\label{eqmuLower}
\mu(s)\geq c\frac{1}{\lambda}\ln \frac{1}{\lambda},\ \ s\in
[0,\ell),
\end{equation}
for $\lambda>0$ sufficiently small. It then follows from
(\ref{eqfthird}) and (\ref{eqw3glob}) that
\begin{equation}\label{eqfUpper}
\left|\partial_s^k f(s) \right|\leq C_k \left(\ln
\frac{1}{\lambda}\right)^{-1},\ \ s\in [0,\ell),\ k\geq 0,
\end{equation}
for $\lambda>0$ sufficiently small (having increased the values of
the generic constants $C_k$ if needed, something that we will do
in the sequel without explicitly mentioning). Therefore, in view
of (\ref{eqDelta1third})-(\ref{eqDelta2third}), we obtain that
\begin{equation}\label{eqDelta1E1estim}
\left|\partial_s^k \Delta_1(s) \right|+\left|\partial_s^k E_1(s)
\right|\leq C_k,
\end{equation}
and
\begin{equation}\label{eqDelta2E2estim}
\left|\partial_s^k \Delta_2(s) \right|+\left|\partial_s^k E_2(s)
\right|\leq C_k\left(\ln \frac{1}{\lambda}\right)^{-2},
\end{equation}
$s\in [0,\ell)$, $k\geq 0$, for $\lambda>0$ small (keep in mind
(\ref{eqw2holder}), (\ref{eqGamma1third}) and (\ref{eqw1third})).
In turn, the functions $A_2$ and $B_2$, as defined in Lemma
\ref{lemphi2}, satisfy
\[
\left|\partial_s^k A_2(s) \right|+\left|\partial_s^k B_2(s)
\right|\leq C_k\left(\ln \frac{1}{\lambda}\right)^{-2},
\]
$s\in [0,\ell)$, $k\geq 0$, for $\lambda>0$ small. It follows from
the above and Lemma \ref{lemphi2} that
\[
\left|\phi_1(s,x)\right|\leq C\left(\ln
\frac{1}{\lambda}\right)^{-1}(x^2+1),\ \ |x|\leq C\ln\left(\ln
\frac{1}{\lambda}\right),\ s\in [0,\ell).
\]
In fact, this estimate can be differentiated arbitrary many times,
that is \[\left|(\phi_1)_x\right|\leq C\left(\ln
\frac{1}{\lambda}\right)^{-1}\left(|x|+1\right),\ \
\left|(\phi_1)_s\right|\leq C\left(\ln
\frac{1}{\lambda}\right)^{-1}\left(x^2+1\right),
\]
and so on. Similarly, we have the estimate
\[
\left|\phi_2(s,x)\right|\leq C\left(\ln
\frac{1}{\lambda}\right)^{-2}\left(|x|^3+1\right),\ \ |x|\leq
C\ln\left(\ln \frac{1}{\lambda}\right),\ s\in [0,\ell),
\]
which can also be differentiated arbitrary many times.

Armed with the above information, and keeping the asymptotic
behavior of $V_0$ in mind, it is easy to verify the following
proposition which represents the main result of this section.

\begin{proposition}\label{proRemainderInner}
The inner approximation $u_2$, defined in (\ref{equ2}) (with
$\phi=\phi_2$ as in Lemma \ref{lemphi2}), satisfies
\[
\left|\Delta u_2(y)+\lambda^2 e^{u_2(y)} \right|\leq C
\left[\ln\left(\ln \frac{1}{\lambda} \right) \right]^2\left( \ln
\frac{1}{\lambda}\right)^{-1},
\]
provided that there is a (different) constant $C>0$ such that
\[
\textrm{dist}(y,\gamma)\leq C \left[\ln\left(\ln \frac{1}{\lambda}
\right) \right]\left( \ln \frac{1}{\lambda}\right)^{-1},
\]
for $\lambda>0$  sufficiently small.
\end{proposition}

\section{The remainder of the outer approximations}
Here, we will estimate the remainder that is left in the equation
of (\ref{eqEqradialmain}) by the outer approximations $W^\pm_2$.
\begin{proposition}\label{proRemainderOuter}
The outer approximations $W^\pm_2$, defined in (\ref{eqWpm}),
satisfy
\[
\left|\Delta W^\pm_2(y)+\lambda^2 e^{W^\pm_2(y)} \right|\leq
\left( \ln \frac{1}{\lambda}\right)^{-cM},
\]
if \begin{equation}\label{eqdescrreg}y\in \Omega^\pm \
\textrm{and}\
                  \textrm{dist}(y,\gamma)\geq 2M \left[\ln\left(\ln
\frac{1}{\lambda} \right) \right]\left(\ln
\frac{1}{\lambda}\right)^{-1},\end{equation} where $c>0$ is
independent of \emph{both} $\lambda>0$ and $M>1$, provided that
$\lambda>0$ is sufficiently small.
\end{proposition}
\begin{proof}
In view of (\ref{eqW<1}) and (\ref{eqHopf}), we have that
\[
W_\gamma^\pm\leq 1-cM \left[\ln \left(\ln \frac{1}{\lambda}
\right) \right] \left(\ln \frac{1}{\lambda} \right)^{-1}
\]
in the corresponding regions that are described by
(\ref{eqdescrreg}), with the constant $c>0$ independent of
\emph{both} $\lambda>0$ and $M>1$, provided that $\lambda>0$ is
sufficiently small. So, in these regions, recalling
(\ref{eqgamma1+}) and that $\Gamma_1=\Gamma_2$, we obtain that
\[
\Gamma_2 W_\gamma^\pm\leq 2 \ln \frac{1}{\lambda}-cM \ln \left(\ln
\frac{1}{\lambda}\right),
\]
for a possibly different constant $c>0$ (still independent of both
$\lambda$ and $M$), provided that $\lambda>0$ is sufficiently
small.

Since $W_2^\pm$ are harmonic and all the $w^\pm_i$'s are uniformly
bounded in $\lambda$ (and independent of $M$), we find that in
each respective region of (\ref{eqdescrreg}) it holds that
\[
\left| \Delta W^\pm_2+\lambda^2 e^{W^\pm_2}\right|=\lambda^2
e^{W^\pm_2}\leq C \lambda^2 e^{\Gamma_2 W^\pm_\gamma}\leq C
\left(\ln \frac{1}{\lambda} \right)^{-cM},
\]
for some constants $c,C>0$ that are independent of both $\lambda$
and $M$, provided that $\lambda>0$ is sufficiently small, as
desired.
\end{proof}

\section{Patching the inner and outer approximations}

In this section, we will construct a global smooth approximate
solution to the problem (\ref{eqEqradialmain}), for $\lambda>0$
small, by interpolating between the inner approximation $u_2$ and
the outer ones $W^\pm_2$ at a distance of order
$\left[\ln\left(\ln \frac{1}{\lambda} \right) \right]\left( \ln
\frac{1}{\lambda}\right)^{-1}$ from the curve $\gamma$. To this
aim, we will need to estimate the differences $u_2-W_2^\pm$,
$\left(u_2-W_2^\pm\right)_t$ and
$\Delta\left(u_2-W_2^\pm\right)=\Delta u_2$ in the respective
interpolation regions on each side of $\gamma$.

\subsection{The estimate for $u_2-W_2^\pm$}
In view of (\ref{eqw0third}), (\ref{eqGamma1third}),
(\ref{eqmuthird}), (\ref{eqw1third}), (\ref{eqw2holder}),
(\ref{eqw3glob}), (\ref{eqRestim}), (\ref{eqmuUpper}),
(\ref{eqmuLower}), (\ref{eqfUpper}) and (\ref{eqDelta1E1estim}),
it follows from (\ref{equ2-W2left}) and (\ref{equ2-W2right}) that
\begin{equation}\label{equ2-W2final}
\left|\left(u_2-W_2^\pm\right) (s,\mp t) \right|\leq  C
\left[\ln\left(\ln \frac{1}{\lambda} \right) \right]^4\left( \ln
\frac{1}{\lambda}\right)^{-3},
\end{equation}
for $s\in [0,\ell)$ and $M  \left[\ln\left(\ln \frac{1}{\lambda}
\right) \right]\left( \ln \frac{1}{\lambda}\right)^{-1}\leq t\leq
2M \left[\ln\left(\ln \frac{1}{\lambda} \right) \right]\left( \ln
\frac{1}{\lambda}\right)^{-1}$, provided that $\lambda>0$ is
sufficiently small (having increased the value of $M$ if
necessary).

\subsection{The estimate for $\left(u_2-W_2^\pm\right)_t$}
It follows from (\ref{eqasymptoticU0U0s}), Lemma \ref{lemphi1},
(\ref{equ2}), and Lemma \ref{lemphi2} that
\[
\begin{array}{rcl}
  (u_2)_x & =  &
(V_0)_x(x)+(\phi_1)_x(s,x)+(\phi_2)_x(s,x)\\
& &\\
  & =&3\left(-\frac{\sqrt{2}}{6}\lambda^{-2}\mu''\mu^{-3}+\frac{\sqrt{2}}{3}\lambda^{-2}\mu^{-2}k^2 \right)x^2
    +2\left(-\lambda^{-2}\mu'' \mu^{-3}+\lambda^{-2}(\mu')^2\mu^{-4}+k^2\lambda^{-2}\mu^{-2}\right. \\
    &   &   \\
    &   &  \left. +\frac{\sqrt{2}}{2}k\lambda^{-1}\mu^{-1}
    -\frac{1}{\sqrt{2}}\lambda^{-2}\mu^{-2}\Delta_1k+\sqrt{2}\lambda^{-1}\mu^{-2}\mu'f'+\frac{1}{\sqrt{2}}\lambda^{-1}\mu^{-1}f''+\frac{1}{\sqrt{2}}
    \lambda^{-1}\mu^{-1}k^2 f \right)x \\
    &   &   \\
    &   &+2k
    \lambda^{-1}\mu^{-1}-\sqrt{2}\lambda^{-1}\mu^{-1}\Delta_1-\sqrt{2}\Delta_2+B_2+\sqrt{2}+\sum_{i=1}^{2}\left(|\Delta_i|+|E_i|+1
    \right)^2 \mathcal{O}(e^{cx}),
\end{array}
\]
as $x\to -\infty$. In turn, it follows from the definition of the
coordinate $x$ in (\ref{eqv}) that
\[
\begin{array}{rcl}
  (u_2)_t & = & \left(-\frac{1}{\sqrt{2}}\lambda\mu''+\sqrt{2}\lambda\mu k^2 \right)t^2+\left(\sqrt{2}\lambda \mu'' f-2\sqrt{2}\lambda\mu f k^2-2\mu''\mu^{-1}+2(\mu')^2\mu^{-2}+2k^2 \right. \\
    &   &   \\
    &   & \left. +\sqrt{2}k\lambda\mu-\sqrt{2}\Delta_1k+2\sqrt{2}\lambda\mu' f'+\sqrt{2}\lambda\mu f''+\sqrt{2}\lambda\mu k^2 f \right)t-\frac{1}{\sqrt{2}}\lambda\mu'' f^2+\sqrt{2}\lambda\mu f^2 k^2 \\
    &   &   \\
    &   & +2f\mu'' \mu^{-1}-2f (\mu')^2\mu^{-2}-2k^2 f-\sqrt{2}k\lambda \mu f+\sqrt{2}f\Delta_1 k-2\sqrt{2}\lambda f\mu' f'-\sqrt{2}\lambda\mu f f''  \\
    &   &   \\
    &   &-\sqrt{2}\lambda\mu k^2 f^2+2k-\sqrt{2}\Delta_1-\sqrt{2}\Delta_2\lambda\mu+B_2\lambda\mu+\sqrt{2}\lambda\mu   \\
    &   &   \\
    &   & +\lambda\mu\sum_{i=1}^{2}\left(|\Delta_i|+|E_i|+1
    \right)^2 \mathcal{O}(e^{c\lambda\mu(t-f)}),
\end{array}
\]
for $s\in [0,\ell)$ and $-C  \left[\ln\left(\ln \frac{1}{\lambda}
\right) \right]\left( \ln \frac{1}{\lambda}\right)^{-1}\leq t\leq
-c \left[\ln\left(\ln \frac{1}{\lambda} \right) \right]\left( \ln
\frac{1}{\lambda}\right)^{-1}$. Analogously, we have that
\[
\begin{array}{rcl}
  (u_2)_t & = & \left(\frac{1}{\sqrt{2}}\lambda\mu''-\sqrt{2}\lambda\mu k^2 \right)t^2+\left(-\sqrt{2}\lambda \mu'' f+2\sqrt{2}\lambda\mu f k^2-2\mu''\mu^{-1}+2(\mu')^2\mu^{-2}+2k^2 \right. \\
    &   &   \\
    &   & \left. -\sqrt{2}k\lambda\mu+\sqrt{2}\Delta_1k-2\sqrt{2}\lambda\mu' f'-\sqrt{2}\lambda\mu f''-\sqrt{2}\lambda\mu k^2 f \right)t+\frac{1}{\sqrt{2}}\lambda\mu'' f^2-\sqrt{2}\lambda\mu f^2 k^2 \\
    &   &   \\
    &   & +2f\mu'' \mu^{-1}-2f (\mu')^2\mu^{-2}-2k^2 f+\sqrt{2}k\lambda \mu f-\sqrt{2}f\Delta_1 k+2\sqrt{2}\lambda f\mu' f'+\sqrt{2}\lambda\mu f f''  \\
    &   &   \\
    &   &+\sqrt{2}\lambda\mu k^2 f^2+2k+\sqrt{2}\Delta_1+\sqrt{2}\Delta_2\lambda\mu+B_2\lambda\mu-\sqrt{2}\lambda\mu   \\
    &   &   \\
    &   & +\lambda\mu\sum_{i=1}^{2}\left(|\Delta_i|+|E_i|+1
    \right)^2 \mathcal{O}(e^{-c\lambda\mu(t-f)}),
\end{array}
\]
for $s\in [0,\ell)$ and $c  \left[\ln\left(\ln \frac{1}{\lambda}
\right) \right]\left( \ln \frac{1}{\lambda}\right)^{-1}\leq t\leq
C \left[\ln\left(\ln \frac{1}{\lambda} \right) \right]\left( \ln
\frac{1}{\lambda}\right)^{-1}$.

%\[
%(w_1^\pm)_s(s,0)=2\mu'\mu^{-1},
%\]
%\[
%(w_1^\pm)_{ss}(s,0)=2\mu''\mu^{-1}-2(\mu')^2\mu^{-2},
%\]
%\[
%(w_1^\pm)_{tt}(s,0)+2\mu''\mu^{-1}-2(\mu')^2\mu^{-2}\pm
%\sqrt{2}k^2\lambda\mu f\pm \sqrt{2}k\Delta_1=0,
%\]
%\[\begin{array}{rl}
%    0= &  (w_0^\pm+w_3^\pm)_{tt}(s,0)\mp \sqrt{2}\lambda\mu''f\mp
%2\sqrt{2}\lambda\mu'f'\mp \sqrt{2}\lambda\mu
%f''-2k^2-2k\mu^{-1}\mu''f+2(\mu')^2k\mu^{-2}f
% \\
%      &       \\
%      & +2k^3f-b_1kf''-b_2k\mu^{-1} \mu'f'-b_3\lambda^{-1}\mu^{-1}\Delta_1 k^2-b_4
%k^3 f-b_5\lambda^{-1}\mu^{-1}\Delta_1 E_1 k-kR(f,w_3^\pm),
%  \end{array}
%\]

On the other side, it follows from (\ref{eqWpm}),
(\ref{eqw0third}), (\ref{eqmuthird}), (\ref{eqDelta1third}),
(\ref{eqDelta2third}), (\ref{eqHartt}), (\ref{eqHarttt}),
(\ref{eqw0w3tt}) and (\ref{eqw1tt}) that
\[
\begin{array}{rcl}
  (W_2^\pm)_t(s,t) & = & \Gamma_2 (W_\gamma^\pm)_t(s,t)+\sum_{i=0}^{4}(w_i^\pm)_t(s,t) \\
    &   &   \\
    &  = & \pm \sqrt{2}\lambda\mu+2k \mp \sqrt{2}k\lambda\mu f\mp\sqrt{2}\Delta_1
\mp \sqrt{2}\lambda\mu \Delta_2\mp
\frac{1}{\sqrt{2}}\lambda\mu''f^2\pm \sqrt{2}\Delta_1 k f
 \\
    &   &   \\
    &   & \mp 2 \sqrt{2}\lambda\mu'f'f\mp \sqrt{2}\lambda\mu f''f+\partial_n w_3^\pm(s,0)\pm \partial_n w_4^\pm(s,0)+t\left(\pm \sqrt{2}k\lambda\mu \pm \sqrt{2}\lambda\mu''f \right.   \\
    &   &   \\
    &   &\left.\pm \sqrt{2}\lambda\mu f''\pm 2 \sqrt{2}\lambda\mu'f'+2k^2-2\mu''\mu^{-1}+2(\mu')^2\mu^{-2}\mp \sqrt{2}k^2\lambda\mu f\mp \sqrt{2}k\Delta_1 \right.   \\
    &   &   \\
    &   &\left.+(w_2^\pm)_{tt}(s,0) +(w_4^\pm)_{tt}(s,0) \right)+\frac{t^2}{2}\left(\mp \sqrt{2}\lambda \mu''\pm 2\sqrt{2}k^2 \lambda\mu+\sum_{i=0}^{4}
    (w_i^\pm)_{ttt}(s,0) \right)   \\
    &   &\\
    & &+\mathcal{O}\left(t^3 \|(W_2^\pm)_{tttt}\|_{L^\infty(\Omega^\pm)}
    \right),
\end{array}
\]
for $s\in [0,\ell)$ and $0\leq \mp t\leq C \left[\ln\left(\ln
\frac{1}{\lambda} \right) \right]\left( \ln
\frac{1}{\lambda}\right)^{-1}$.

By the above relations and (\ref{eqw3third})-(\ref{eqRw3}), we
deduce that
\[
\begin{array}{rcl}
 (u_2-W_2^\pm)_t & = & -R(f,w_3^\pm)\lambda\mu\mp\partial_nw_4^\pm(s,0)+t\left(-2k\mu^{-1}\mu''f+2(\mu')^2k \mu^{-2}f+2k^3 f-b_1k f'' \right. \\
    &   &   \\
    &   &  \left. -b_2 k \mu^{-1}\mu'f'-b_3\lambda^{-1}\mu^{-1}\Delta_1k^2-b_4 k^3 f-b_5\lambda^{-1}\mu^{-1}\Delta_1
    E_1k-kR(f,w_3^\pm)\right. \\
   &   &   \\
    &   & \left. -
    (w_2^\pm)_{tt}(s,0)-(w_4^\pm)_{tt}(s,0)\right)-\frac{t^2}{2}\sum_{i=0}^{4}(w_i^\pm)_{ttt}(s,0)  \\
    &   &   \\
   &  & +\mathcal{O}\left(t^3 \|(W_2^\pm)_{tttt}\|_{L^\infty(\Omega^\pm)} \right)+\lambda\mu\sum_{i=1}^{2}\left(|\Delta_i|+|E_i|+1
    \right)^2 \mathcal{O}(e^{-c\lambda\mu|t-f|}),
\end{array}
\]
for $s\in [0,\ell)$ and $c \left[\ln\left(\ln \frac{1}{\lambda}
\right) \right]\left( \ln \frac{1}{\lambda}\right)^{-1}\leq \mp
t\leq C \left[\ln\left(\ln \frac{1}{\lambda} \right) \right]\left(
\ln \frac{1}{\lambda}\right)^{-1}$.

Before proceeding further, let us complete the estimates for the
harmonic functions $w_i^\pm$, $i=0,\cdots,4$, by estimating
$w_4^\pm$. It follows from (\ref{eqw4third}), (\ref{eqmuUpper}),
(\ref{eqmuLower}), (\ref{eqfUpper}), (\ref{eqDelta1E1estim}) and
(\ref{eqDelta2E2estim}) that
\[
\left|w_4^\pm \right|\leq C \left(\ln \frac{1}{\lambda}
\right)^{-2}\ \ \ \textrm{on}\ \ \gamma.
\]
In turn, by the same argument leading to (\ref{eqw2holder}), we
find that
\begin{equation}\label{eqw4holder}
\|w_4^\pm\|_{C^{m}(\bar{\Omega}^\pm)}\leq
C_m\left(\ln\frac{1}{\lambda} \right)^{-2},\ \ m\geq 0.
\end{equation}

As before, it follows readily that
\begin{equation}\label{eq(u2-W2)tfinal}
\left|\left(u_2-W_2^\pm\right)_t (s,\mp t) \right|\leq  C
\left[\ln\left(\ln \frac{1}{\lambda} \right) \right]^3\left( \ln
\frac{1}{\lambda}\right)^{-2},
\end{equation}
for $s\in [0,\ell)$ and $M  \left[\ln\left(\ln \frac{1}{\lambda}
\right) \right]\left( \ln \frac{1}{\lambda}\right)^{-1}\leq t\leq
2M \left[\ln\left(\ln \frac{1}{\lambda} \right) \right]\left( \ln
\frac{1}{\lambda}\right)^{-1}$, provided that $\lambda>0$ is
sufficiently small (having increased the value of $M$ if
necessary).

\subsection{The estimate for $\Delta\left(u_2-W_2^\pm\right)=\Delta u_2$}
It follows immediately from (\ref{equ2}),
(\ref{eqasymptoticU0U0s}), Lemmas \ref{lemphi1}-\ref{lemphi2},
(\ref{eqmuLower}), (\ref{eqfUpper}) and Proposition
\ref{proRemainderInner} that
\begin{equation}\label{eqLaplacianu2}
\left|\Delta u_2(y)\right|\leq C \left[\ln\left(\ln
\frac{1}{\lambda} \right) \right]^2\left( \ln
\frac{1}{\lambda}\right)^{-1},
\end{equation}
if \begin{equation}\label{eqinterpolRegion}M \left[\ln\left(\ln
\frac{1}{\lambda} \right) \right]\left( \ln
\frac{1}{\lambda}\right)^{-1} \leq\textrm{dist}(y,\gamma)\leq 2M
\left[\ln\left(\ln \frac{1}{\lambda} \right) \right]\left( \ln
\frac{1}{\lambda}\right)^{-1},
\end{equation}
for $\lambda>0$  sufficiently small (having possibly increased
$M$, since in this region $\lambda^2e^{u_2}$ is of order
$\left(\ln \frac{1}{\lambda} \right)^{-cM}$ with $c>0$ independent
of \emph{both} $M$ and $\lambda$).
\subsection{The global approximate solution $u_{ap}$}
We are now in position to construct a smooth global approximate
solution to the problem (\ref{eqEqradialmain}), by interpolating
between the inner and outer approximations in the region described
by (\ref{eqinterpolRegion}), and to be able to estimate the
remainder that is left by it in the equation. As expected, this
task will require us to use some cutoff functions.

Consider a fixed smooth cutoff function such that
\begin{equation}\label{eqetaCutoff}
\eta(\tau)=\left\{\begin{array}{ll}
                 0, & |\tau|\leq 1, \\
                   &   \\
                 1, & |\tau|\geq 2.
               \end{array}
\right.
\end{equation}
Then, let
\[
\eta_\lambda(t)=\eta\left( \frac{t}{M \left[\ln\left(\ln
\frac{1}{\lambda} \right) \right]\left( \ln
\frac{1}{\lambda}\right)^{-1}}\right).
\]

We can now
 define our global approximate solution, using the local
coordinates $(s,t)$, as
\[
u_{ap}(y)=\left\{ \begin{array}{ll}
                  u_2, & \textrm{dist}(y,\gamma)\leq M \left[\ln\left(\ln
\frac{1}{\lambda} \right) \right]\left( \ln
\frac{1}{\lambda}\right)^{-1}, \\
                    &   \\
                  u_2+\eta_\lambda(t)\left(W_2^\pm-u_2 \right), & \mp t\in \left(M \left[\ln\left(\ln
\frac{1}{\lambda} \right) \right]\left( \ln
\frac{1}{\lambda}\right)^{-1},2M \left[\ln\left(\ln
\frac{1}{\lambda} \right) \right]\left( \ln
\frac{1}{\lambda}\right)^{-1}\right),  \\
                    &   \\
                  W^\pm_2,  & y\in \Omega^\pm \ \textrm{and}\
                  \textrm{dist}(y,\gamma)\geq 2M \left[\ln\left(\ln
\frac{1}{\lambda} \right) \right]\left(\ln
\frac{1}{\lambda}\right)^{-1}.
                \end{array}
\right.
\]

The main result of this section is the following.
\begin{proposition}\label{proRemainderGlobal}
We can choose a large $M$ such that the global approximation
$u_{ap}$ satisfies
\[
\|\Delta u_{ap}+\lambda^2 e^{u_{ap}} \|_{L^\infty(\Omega)}\leq C
\left[\ln\left(\ln \frac{1}{\lambda} \right) \right]^2\left( \ln
\frac{1}{\lambda}\right)^{-1},
\]
if $\lambda>0$ is sufficiently small.
\end{proposition}
\begin{proof}
By virtue of Propositions \ref{proRemainderInner} and
\ref{proRemainderOuter}, it remains to consider the intermediate
regions in (\ref{eqinterpolRegion}). There, using
(\ref{eqgradFermi}) and the fact that $W_2^\pm$ are harmonic, the
remainder under consideration reduces to
\[\Delta u_2+(\Delta
\eta_\lambda)\left(W_2^\pm-u_2 \right)+2\partial_t \eta_\lambda
\left(W_2^\pm-u_2 \right)_t-\eta_\lambda \Delta u_2 +\lambda^2
e^{u_2+\eta_\lambda (W_2^\pm-u_2)}.
\]
The sought after estimate now follows readily by using
(\ref{equ2-W2final}), (\ref{eq(u2-W2)tfinal}),
(\ref{eqLaplacianu2}), the direct estimates
\[
\left|\partial_t \eta_\lambda \right|\leq C \left[\ln\left(\ln
\frac{1}{\lambda} \right) \right]^{-1}\left( \ln
\frac{1}{\lambda}\right),\ \ \left|\Delta \eta_\lambda \right|\leq
C \left[\ln\left(\ln \frac{1}{\lambda} \right) \right]^{-2}\left(
\ln \frac{1}{\lambda}\right)^2\ \ (\textrm{keep in mind
(\ref{eqS(u)})}),
\]
and the comment below (\ref{eqinterpolRegion}) (to estimate the
last term).
\end{proof}

 \textbf{Acknowledgments.} The idea of using conformal invariance in order to get the existence of the curve $\gamma$ in Proposition \ref{proW} is borrowed from a personal communication of the author with M. Kowalczyk.

\end{document}